\documentclass[11pt,reqno,a4paper]{amsart}

\usepackage{amsfonts}
\usepackage{epsfig}
\usepackage{graphicx}
\usepackage{amsmath}
\usepackage{amssymb}
\usepackage{color}
\usepackage{mathrsfs}
\usepackage{txfonts}

\makeatletter \@addtoreset{equation}{section} \makeatother

\renewcommand\thetable{\thesection.\@arabic\c@table}

\theoremstyle{plain}

\newtheorem{theorem}{Theorem }[section]
\newtheorem{proposition}{Proposition}[section]

\newtheorem{lemma}{Lemma}[section]
\newtheorem{corollary}{Corollary}[section]

\theoremstyle{definition} \theoremstyle{remark}
\newtheorem{remark}[theorem]{Remark}
\newtheorem{example}[theorem]{Example}
\newtheorem{definition}[theorem]{Definition}

\newcommand{\vep}{\varepsilon}

\newcommand{\diam}{\operatorname{diam}}

\newcommand{\cP}{\mathcal{P}}

\begin{document}

\title[Specification and thermodynamical properties of semigroup actions]
{Specification and thermodynamical properties\\ of semigroup actions}

\author{Fagner B. Rodrigues}
\address{Departamento de Matem\'atica, Universidade Federal do Rio Grande do Sul, Brazil. \& CMUP, University of Porto, Portugal}
\email{fagnerbernardini@gmail.com}

\author{Paulo Varandas}
\address{Departamento de Matem\'atica, Universidade Federal da Bahia\\
  Av. Ademar de Barros s/n, 40170-110 Salvador, Brazil. \& CMUP, University of Porto, Portugal}
\email{paulo.varandas@ufba.br}

\date{\today}

\begin{abstract}
In the present paper we study the thermodynamical properties of
finitely generated continuous subgroup actions. We propose a notion of topological
entropy and pressure functions that does not depend on the growth rate of the semigroup
and introduce strong and orbital specification properties, under which, the semigroup
actions have positive topological entropy and all points are entropy points.
Moreover, we study the convergence and Lipschitz regularity of the pressure function
and obtain relations between topological entropy and exponential growth rate of periodic
points in the context of semigroups of expanding maps, obtaining a partial extension of the results
obtained by Ruelle for $\mathbb Z^d$-actions~\cite{Ruelle} .
The specification properties for semigroup actions and
the corresponding one for its generators and the action of push-forward maps
is also discussed.
\end{abstract}

\maketitle

\section{\label{sec:level1}Introduction}

The thermodynamical formalism was brought from statistical mechanics to dynamical systems
by the pioneering works of Sinai, Ruelle and Bowen \cite{Bo75,BR75,Si72, RuelleS} in the mid seventies.
The correspondance between one-dimensional lattices and uniformly hyperbolic maps allowed
to translate and introduce several notions of Gibbs measures and equilibrium states in the realm
of dynamical systems. The present study of the thermodynamical formalism for non-uniformly
hyperbolic dynamical systems is now paralel to the development of a thermodynamical formalism
of gases with infinitely many states, a hard subject not yet completely understood. Moreover,
the notion of entropy constitutes one of the most important in the study of dynamical systems
(we refer the reader to Katok~\cite{Katok} and references therein for a survey on the state of the art).

An extension of the thermodynamical formalism for continuous finitely generated group actions
has revealed fundamental difficulties and the global description of the theory is still incomplete.
A first attempt was to consider continuous actions associated to finitely generated abelian groups.
The statistical mechanics of expansive $\mathbb{Z}^d$-actions
satisfying a specification property was studied by Ruelle~\cite{Ruelle}, where he introduced
a pressure function, defined on the space of continuous functions, and discussed its relations with measure
theoretical entropy and free energy.
The notion of specification was introduced in the seventies as a property of uniformly hyperbolic
basic pieces and became a characterization of complexity in dynamical systems.
The crucial fact that continuous $\mathbb{Z}^d$-actions
on compact spaces admit probability measures invariant by every continuous maps associated to
the group action, allowed Ruelle to prove a variational principle for the topological pressure and
to build equilibrium states as the class of pressure maximizing invariant probability measures.
This duality between topological and measure theoretical complexity of the dynamical system
is very fruitfull, e.g. was used later by Eizenberg, Kifer and Weiss~\cite{EKW}  to establish large
deviations principles to $\mathbb{Z}^d$-actions satisfying the specification property.
Other specification properties of interest have been introduced recently
(see e.g.~\cite{ChuLi,Var10}).

A unified approach to the thermodynamical formalism of continuous group actions is still unavailable,
while still few definitions of topological pressure exists and most of them unrelated.
Moreover the connection between topological and ergodic properties of group actions still
fails to provide a complete description the complexity of the dynamical system.
In many cases the existent definitions for topological entropy take into account either abelianity, amenability
or growth rate of the corresponding group. A non-extensive list of contributions by many authors
include important contributions by Ghys, Langevin, Walczak, Friedland, Lind, Schmidt,
Bufetov, Bi\'s, Urbanski, Ma, Wu, Miles, Ward, Chen, Zheng and Schneider among others
(see e.g. \cite{GLW,Fried,LiSc, Bufetov,Bis, BiU, MW, BiW,BisII, MW14, CZ14,Sc15} and references therein).

Our main goal here is to describe the topological aspects of the thermodynamical formalism for semigroup actions
for general finitely generated semigroup actions, where no commutativity or conditions on the
semigroup growth rate are required. Inspired by a notion of topological entropy of free semigroups
by Bufetov~\cite{Bufetov},
given a finitely generated semigroup $(G, G_1)$ where $G_1=\{id, g_1, \dots, g_m\}$ is a set of generators
we consider the coding
\begin{equation}\label{eq:coding}
\begin{array}{cccc}
\iota : & F_m & \to & G \\
	& i_n \dots i_1 & \mapsto & g_{i_n} \circ \dots \circ g_{i_1}
\end{array}
\end{equation}
where $F_m$ denotes the free semigroup with $m$ elements. This coding 
is injective if and only if $G$ is a free semigroup. Our thermodynamical approach for the semigroup action is to
average the complexity of each dynamics $g\in G$ with a weight corresponding to the size of $\iota^{-1}(g)$, that is,
how often a particular semigroup element $g$ arises by concatenation of the generators.

 E.g. if all generators commute and do not have finite order then $G \simeq \mathbb Z^m$ and every element in $G$ has the same weight, a property that will change
substantially in the case of semigroups of exponential growth with a non-trivial abelian subgroup.
This approach has the advantage of being independent of the semigroup growth rate, hence to propose a unified approach to
the study of semigroups with substantially different growth rates (see Section~\ref{section of examples} for examples) and the disadvantage to depend \emph{a priori} on the set of generators for the semigroup.
Inspired by several forms of the specification property for discrete time transformations with some
hyperbolicity (see e.g. \cite{S,PS,SSY10,OT11,Var10}), we also introduce some notions of strong and orbital specification properties
for continuous actions associated to finitely generated (not necessarily abelian) groups which are of independent interest.
In the particular case of semigroups $(G,G_1)$ of expanding maps our main contributions can be summarized as follows:
\begin{itemize}
\item[(a)] we introduce a notion of topological pressure $P_{top}((G,G_1),\varphi, X)$ which in independent
	of the semigroup growth rate;
\item[(b)] we prove that the orbital specification properties hold and, consequently, the local complexity at every neighborhood
	of any point coincides with the topological pressure of the dynamical system (see the notions of `entropy point' in
	Subsection~\ref{subsec:entropy});
\item[(c)] using expansiveness, we prove that topological pressure can be computed at a finite scale (omitting a limit in the
	original definition)
\item[(d)] we prove that the topological pressure function  $t\mapsto P_{top}((G,G_1),t \varphi, X)$ for H\"older continuous observables
	$\varphi$ is a uniform limit of $C^1$ functions, hence it is Lipschitz and differentiable Lebesgue almost everywhere; and
\item[(d)] the exponential mean growth of periodic points 
	is bounded from below by topological entropy $P_{top}((G,G_1),\varphi, X)$.
\end{itemize}
In \cite{Ruelle}, Ruelle  studied expansive $\mathbb Z^d$-actions with specification property and obtained that the topological
pressure function is smooth, existence and uniqueness of equilibrium states.
Here we obtained the Lebesgue almost everywhere differentiability of the pressure function for semigroups of expanding maps
that may have exponential growth.  To the best of our knowledge these are the first results after \cite{Ruelle}
(that considered $\mathbb Z^d$-actions) where there are partial results on the the differentiability of the topological pressure
function for group or semigroup actions.

Finally we observe that this is the first part of a program to describe the thermodynamical
properties of semigroup actions following the program of Ruelle~\cite{Ruelle}, and the construction of relevant stationary measures
that describe the ergodic theory of finitely generated semigroup actions of expanding maps will appear elsewhere~\cite{CRV}.
The relation between orbital specification properties for the group action
is also discussed and a class of examples of group actions is given where orbital specification properties
present a flavor of the non-uniform versions arising in non-uniformly hyperbolic dynamics.
In fact, we also study semigroups with non-expanding elements and compare these with the notions of entropy introduced by
Ruelle~\cite{Ruelle} and Ghys, Langevin, Walczak~\cite{GLW}.
For the convenience of the reader, we describe briefly the beginning of each section the
main results to be proved there. Except when we mention explicit otherwise,
we shall consider the context of semigroup actions and, in case the existence of inverse elements
is needed, we shall make precise mention to that fact.
We refer the reader to the statement of the main results
and to Section~\ref{section of examples} for some examples.

This paper is organized as follows. In Section~\ref{section2} we introduce both the strong specification property
and some orbital specification properties for finitely generated semigroups actions and discuss the relation
between these notions and the specification property for the generators. The connections between
specification properties for group actions, for the push-forward group actions
and hyperbolicity are also discussed.

In Section~\ref{sec:thermo} we introduce a notion of topological entropy and pressure for
continuous semigroup actions and study group actions that exhibit some forms of specification. In particular,
we prove that these have positive topological entropy and every point is an entropy point.

In Section~\ref{sec:expanding} we study the semigroup action induced by expanding maps. We prove that
these semigroups satisfy the previous notions of specification and that topological entropy is a lower bound
for the exponential growth rate of periodic orbits. We also deduce that the pressure function acting on the
space of H\"older continuous potentials is Lipschitz, hence almost everywhere differentiable along families
$t\varphi$ with $t\in\mathbb R$ and $\varphi$ H\"older continuous.

Finally, in Section~\ref{section of examples} we provide several examples where we discuss the specification properties and establish a comparison between some notions of topological entropy.

\section{Specification for a finitely generated semigroup actions}\label{section2}

In this section we introduce the notions of specification and orbital specification properties
for the context of group and semigroup actions. The specification property for the group action
implies that all generators satisfy the specification property (Lemma~\ref{esp of elements}) and also
that the push-forward group action satisfies the specification property
(Theorem~\ref{prop:G-esp pf}). Moreover, $C^1$-robust specification implies structural stability
(Corollary~\ref{prop:robust spec}).

\subsection{Strong specification property}

The specification property for a continuous map on a compact metric space $X$
was introduced by Bowen~\cite{Bowen}.
A continuous map $f: X \to X$ satisfies the \emph{specification property} if
for any $\delta>0$ there exists an integer $p(\delta)\geq 1$ such that
the following holds: for every $k\geq 1$, any points $x_1,\dots,
x_k$, and any sequence of positive integers $n_1, \dots, n_k$ and
$p_1, \dots, p_k$ with $p_i \geq p(\delta)$
there exists a point $x$ in $X$ such that
$$
\begin{array}{cc} d\Big(f^j(x),f^j(x_1)\Big) \leq \delta,
	 &\forall \,0\leq j \leq n_1
\end{array}
$$
and
$$
\begin{array}{cc}
d\Big(f^{j+n_1+p_1+\dots +n_{i-1}+p_{i-1}}(x) \;,\; f^j(x_i)\Big)
        \leq \delta &
\end{array}
$$
for every $2\leq i\leq k$ and $0\leq j\leq n_i$. This property means that pieces of orbits of $f$
can be $\delta$-shadowed by a individual orbit provided that the time lag between
each shadowing is larger than some prefixed time $p(\delta)$.

The notion of specification was extended to the context of continuous $\mathbb Z^d$-actions
on a compact metric space $X$ by Ruelle motivated by statistical mechanics.
Let $(\mathbb{Z}^d,+)$ be endowed with the distance
$
d_{\mathbb{Z}^d}(a,b)=\sum_{i=1}^p|a_i-b_i|.
$
Following ~\cite{Ruelle},
the group action $\mathbb Z^d \times X \to X$
satisfies the \emph{specification property} if for any
 $\delta>0$  there exists $p(\delta) > 0$ such that for any finite
families $(\Lambda_i)_{i\in\mathcal{I}}$, $(x_i)_{i\in\mathcal{I}}$ satisfying
if $i \not= j$, the distance of $\Lambda_i$, $\Lambda_j$
(as subsets of $\mathbb{Z}^d$) is $> p(\delta)$,
there is $x \in X$ such that
$d(m_ix,m_ix_i) < \delta$, for all $i \in \mathcal{I}$, and all $m_i\in\Lambda_i$.
This notion clearly extends to group actions associated to finitely generated abelian groups.

\subsubsection*{Specification property for groups and its generators}

In this article we shall address the specification properties and thermodynamical formalism
to deal both with finitely generated  group and  semigroup actions. For simplicity, we shall state our results
in the more general context of  semigroup actions whenever the results do not require the
existence of inverse elements.
More precisely, given a finitely generated semigroup $(G,\circ)$ with a finite set of generators
$G_1=\{id,g_1, g_2, \dots, g_m\}$
one can write
$
G=\bigcup_{n\in \mathbb N_0} G_n
$ where
$
G_0=id
$
and
\begin{equation}\label{eq:semi-group}
\underline g \in G_n \text{ if and only if }
	\underline g=g_{i_n} \dots g_{i_2} g_{i_1}
	\text{ with } g_{i_j} \in G_1
\end{equation}
(where we use $g_j \, g_i$ instead of $g_j\circ g_i$ for notational simplicity).
If, in addition, the elements of $G_1$ are invertible,
the finitely generated group $(G,\circ)$ is defined by
$
G=\bigcup_{n\in \mathbb N_0} G_n
$
where
$
G_0=id, \; G_1=\{id, g_1^{\pm}, g_2^{\pm}, \dots, g_m^{\pm}\}
$
and the elements $\underline g \in G_n$ are defined by \eqref{eq:semi-group}.
In both settings, $G_n$ consists of those group elements which are concatenations of at most
$n$ elements of $G_1$.
Since $id \in G_n$ then $(G_n)_{n\in \mathbb N}$ defines an increasing family
of subsets of $G$.
Moreover, $G$ is a finite semigroup if and only if $G_n$ is empty for every $n$ larger than the cardinality
of the group.
Given a semigroup $G$ we say $g\in G$ has \emph{finite order} if there exists $n\ge 1$ so that $g^n=id$. If the later property does not hold then an element $g\in G$ is said to have \emph{infinite order}.
We say that $\underline g=g_{i_n}\dots g_{i_1}$ is reduced if it is the smaller concatenations of elements of $G_1$ which generates $\underline g$.
Denote by $G_1^*=G_1 \setminus \{id\}$ and $G_n^*=\{\underline g=g_{i_n} \dots g_{i_2} g_{i_1}: g_{i_j}\in G_1^*\}$.
Using the coding function $\iota$ (recall \eqref{eq:coding}) observe $G_n^*=\iota (\{i_n \dots i_1 : i_j\in\{1, \dots, k\})$.

Motivated by applications by actions of  semigroups
we first introduce some generalizations of the previous specification property
for group actions. Let $(G,\circ)$ be a finitely generated group of maps on a compact metric space $X$ endowed  with the distance
$
d_G(h,g)= |h^{-1}g|
$
for $h,g\in G$,
where the right hand side tem is the order of the element $h^{-1}g$ and it is defined by
$|h^{-1}g|:= \inf \{n\ge 1 \colon h^{-1}g \in G_n\}.$
It is not difficult to check that it is a metric in the group $G$
and that $d_G(h,g)=n$ if and only if there exists $\underline g_n \in G_n$ so that
$g=h \, \underline g_n$. We are unaware of a natural notion of metric for semigroups.
The following notion extends of the specification property introduced by \cite{Ruelle} to more general
group actions.

\begin{definition}\label{def:g-specification}
Let $G$ be a finitely generated group, $X$ be a compact metric space
and let $T: G \times X \to X$ be a continuous action. We say that the group action $T$ has
the \emph{specification property} if for any $\delta>0$  there exists $p(\delta) > 0$ such
that for any finite  families $(\Lambda_i)_{i\in\mathcal{I}}$, $(x_i)_{i\in\mathcal{I}}$
so that the $d_G (\Lambda_i, \Lambda_j)>p(\delta)$ for every
$i \not= j$, then there is $x \in X$ such that $d(g_ix,g_ix_i) < \delta$ for every
$i \in \mathcal{I}$ and $g_i\in\Lambda_i$.
\end{definition}

The later notion implies on a strong topological indecomposability of the group action. Given
 a continuous action $T: G \times X \to X$ we say that $T$ is \emph{topologically transitive} if there exists
 a point $x\in X$ such that the orbit $O_G(x):=\{ \underline g(x) : \underline g \in G\}$ is dense in $X$.
We say that $T$ is \emph{topologically mixing} if for any open sets $A,B$ in $X$ there exists $N\ge 1$
such that for any $n\ge N$ there is $\underline g \in G$ with ${\underline g \in G_n^*}$ satisfying $\underline g(A) \cap B \neq \emptyset$. It is easy to check that any continuous action with the specification property is topologically mixing, hence topologically transitive.
For a survey on several mixing properties for group actions
we refer the reader to the survey \cite{CKN} and references therein.

Given a continuous action $T: G\times X \to X$ of a group $G$ on a compact metric space
$X$ we denote, by some abuse of notation, $g : X \to X$ to be the continuous map
$x \mapsto T(g,x)$.
Given $\underline g\in G$ we say that $x\in X$ is a \emph{fixed point} for $\underline g$ if $\underline g(x)=x$ and use the notation $x\in Fix(\underline g)$.
We say that $x \in M$ is a \emph{periodic point of period $n$} if there exists $\underline{g}_n \in G_n$ so that $\underline{g}_n(x) = x$. In other words, $x \in \bigcup_{\underline{g}_n \in G_n}Fix(\underline{g}_n )$.
We let $Per(G_n)$ denote the set of periodic points of period $n$ and set $Per(G)=\bigcup_{n\ge 1} Per(G_n)$.
If the tracing orbit in the specification property can be chosen periodic we will say that the
action satisfies the \emph{periodic specification} property.
It is not hard to check that an invertible  transformation $f: X \to X$ satisfies the specification
property if and only if the group action on $X$ associated to the group $G=\{ f^n : n\in \mathbb Z\}$ (isomorphic to $\mathbb Z$) satisfies the specification property.

The next lemma asserts that this specification property for group actions implies
all generators to satisfy the corresponding property.

\begin{lemma}\label{esp of elements}
Let $G$ be a finitely generated group with generators $G_1=\{g_1^{\pm}, g_2^{\pm},
\dots, g_k^{\pm}\}$. If the group action $T : G\times X \to X$ satisfies the specification
property then every $g\in G_1$ with infinite order has the specification property.
\end{lemma}

\begin{proof}
Let $\delta>0$ be fixed and let $p(\delta)>0$ be given by the specification property for the group action $T$.
Take arbitrary $k\geq 1$, points $x_1,\dots, x_k$, and positive integers $n_1, \dots, n_k$ and
$p_1, \dots, p_k$ with $p_i \geq p(\delta)$. Since $g\in G_1$ is a generator then for any $i=1 \dots k$ the set
$$
\Lambda_i =
		\Big\{g^j \colon \;  \sum_{s=0}^{i-1} ( p_{s}+n_{s} )
		\le j \le   n_i + \sum_{s=0}^{i-1} ( p_{s}+n_{s} )
		\Big\}
$$
is finite and connected (assume $n_0=p_0=0$). Moreover, since $g$ has infinite order
it is not hard to check that
$d_G(\Lambda_i, \Lambda_j)\ge p(\delta)$ for any $i\neq j$. Let
$\bar{x}_j=g^{ -\sum_{s=0}^{j-1} p_{s}+n_{s}}(x_j)$, for $1\leq j\leq k$.
Thus, by the specification property there exists a
point $x \in X$ such that $d(h x, h \bar{x}_i) < \delta$, for all $i =1 \dots k$ and all $h \in\Lambda_i$ which are reduced
in this case to
$$
\begin{array}{cc}
d\Big(g^j(x),g^j(x_1)\Big) \leq \delta,
	 &\forall \,0\leq j \leq n_1
\end{array}
$$
and
$$
\begin{array}{cc}
d\Big(g^{j+n_1+p_1+\dots +n_{i-1}+p_{i-1}}(x) \;,\; g^j(x_i)\Big)
        \leq \delta &
\end{array}
$$
for every $2\leq i\leq k$ and $0\leq j\leq n_i$. This proves that the map $g$ has the specification property
and finishes the proof of the lemma.
\end{proof}

Let us mention that the existence of elements of generators of finite order is not an obstruction
for the group action to have the specification (e.g. the $\mathbb Z^2$-action on $\mathbb T^2=\mathbb R^2/\mathbb Z^2$ whose generators are a hyperbolic automorphism and the reflection on the real axis).
We refer the reader to Section~\ref{section of examples} for a simple example of a $\mathbb Z^2$-action for which the converse implication is not necessarily true.

\subsubsection*{The push-forward group action}

Given a compact metric space $X$ let $\cP(X)$ denote the space of probability measures on $X$,
endowed with the weak$^*$-topology. It is well known that $\cP(X)$ with the weak$^*$ topology is
a compact set. We recall that the weak$^*$-topology in  $\cP(X)$ is metrizable and a metric that generates the topology can be defined as follows. Given a countable dense set of continuous functions $(\phi_k)_{k\ge 1}$ in $C(X)$ and $\mu,\nu\in\cP(X)$ define
$$
d_\cP(\mu,\nu) = \sum_{k\ge 1} \frac1{2^k \|\phi_k\|} \left| \int \phi_k \,d\mu  - \int \phi_k \,d\nu \right|.
$$
For a continuous map $f: X\to X$, the space of $f$-invariant probability measures correspond to the fixed
points of the \emph{push-forward map} $f_\sharp: \cP(X) \to \cP(X)$, which is a continuous map.
For that reason the push-forward $f_\sharp$ reflects the ergodic theoretical aspects of $f$.
Moreover, the dynamics of $f$ is embedded in the one of $f_\sharp$ since it corresponds to the restriction
of $f_\sharp$ to the space $\{\delta_x : x\in X\}\subset \cP(X)$ of Dirac measures on $X$.
This motivates the study of specification properties for the group action of
the push-forward maps.

Given a finitely generated group $G$ and a continuous
group action $T : G \times X \to X$ let us denote by $T_{\sharp} : G \times \cP (X)
\to \cP(X)$ denote the group action defined by $g \cdot \nu = T(g,\cdot)_\sharp \, \nu$.
It is natural to ask wether the specification property can be inherited from this duality relation.

\begin{theorem}\label{prop:G-esp pf}
Let $G$ be a finitely generated group and $T : G \times X \to X$ be
a continuous group action satisfying the specification property. Then the group action
$T_{\sharp} : G \times \cP (X) \to \cP(X)$ satisfies the specification property.
\end{theorem}

The following lemma will play an instrumental role in the proof of the theorem.

\begin{lemma}\label{lemma:approx measure}
Given probability measures $\mu_1,...,\mu_k\in\cP(X)$ and $\delta>0$, there are $N\in\mathbb{N}$
and points $(x_1^i,...,x^i_N)\in X^N$ such that the probabilities
$\mu_i'=\frac{1}{N}\sum_{j=1}^N\delta_{x_j^i}$ satisfy $d(\mu_ i,\mu_i')<\delta$
for $1\leq i\leq k$.
\end{lemma}

\begin{proof}
It is well known that the finitely supported atomic measures are dense in $\cP(X)$. Then,
for $\delta>0$, there are $\bar{\mu}_1,...,\bar{\mu}_k$, with
$\bar\mu_j=\sum_{j=1}^M\alpha_i^j\delta_{x_i^j}\in\cP(X)$, so that
$d(\mu_k,\bar\mu_k)<\delta\slash2$. Let $p_i^j\slash q_i^j$ be
a positive rational such that $|\alpha_i^j-p_i^j\slash q_i^j|<\delta\slash 10$.
Let $N=\prod_{i,j=1}^Mq_i^j$ and $ N_k^j=p^j_k \prod_{i,j=1, i\not=k}^Mq_i^j$. Notice that
$|N_k^j / N -\alpha_k^j| <\delta\slash 10$ and
$$
\mu_j'=\frac{1}{N}\left(\sum_{i=1}^{N_1^j}\delta_{x_1^j}+\sum_{i=1}^{N_2^j}\delta_{x_2^j}+...+\sum_{i=1}^{N_k^j}\delta_{x_k^j}\right),
$$
satisfies $d(\mu'_j,\bar\mu_j)<\delta\slash2$, and by triangular inequality, $d(\mu_j,\mu_j')<\delta$.
\end{proof}

\begin{proof}[Proof of the Theorem~\ref{prop:G-esp pf}]

Assume that the action $T : G \times X \to X$ has the
specification property. Clearly, if $T$ satisfies the specification property then for any $N\ge 1$ the continuous
action $T^{(N)} : G \times X^N \to X^N$ on the product space $X^N$
endowed with the distance $d_N((x_i)_i, (y_i)_i)=\max_{1\le i \le N} d(x_i, y_i)$ and
given by $g \cdot (x_1, \dots , x_N) = (g x_1, \dots , g x_N)$ also satisfies the specification. In fact,
for any $\delta>0$ just take $p(\delta)>0$ as given by the specification property for $T$.

Let us proceed with the proof of the theorem. Take $\delta>0$ and let $p(\delta\slash2)$ be given by the specification property. Take $\mu_1,...,\mu_k\in\cP(X)$ and $\Lambda_1,\dots, \Lambda_k$ finite subsets
of $G$ with $d(\Lambda_i,\Lambda_j)>p(\delta\slash2)$.  Let $\mu'_i=\frac{1}{N}\sum_{j=1}^N\delta_{x_j^i}$,
such that $d(g \mu_i',g \mu_i)<\delta\slash2$ for all $g \in\Lambda_i$.
By considering the finite sequence $(x_1^i,...,x^i_N)_{i=1}^k\subset X^N$
and the sets $\Lambda_1,..\Lambda_k$, there exists a point $(x_1,...,x_N) \in X^N$ in the product space
such that
$$
d(g \cdot (x_1,...,x_N),g \cdot (x_1^i,...,x^i_N))<\frac{\delta}{2} \mbox{ for all } g \in\Lambda_i.
$$
It implies that the probability measure $\mu=\frac{1}{N}\sum_{j=1}^N\delta_{x_j}$
satisfies
$$
d( g \cdot \mu_i, g \cdot \mu)
	\leq d( g \cdot \mu_i',g \cdot \mu)
	+d( g \cdot \mu_i',g\cdot \mu_i)
	<\delta, \mbox{ for all } g\in\Lambda_i.
$$
This completes the proof of the theorem.
\end{proof}

The converse implication in the previous theorem is not immediate. In fact, given the specification property
for $T_\sharp$ and any specified pieces of orbit by $T_\sharp$ it is not clear that this can be shadowed by
the $T_\sharp$-orbit of a Dirac probability measure $\delta_x$.
Nevertheless this is indeed the case for the dynamics of continuous interval maps.

\begin{corollary}
Let $f$ be a continuous interval map. Then $f$ satisfies the specification
property if and only  if $f_\sharp$ satisfies the specification property.
\end{corollary}

\begin{proof}
It follows from Theorem~\ref{prop:G-esp pf} that the specification property for $f$
implies the specification property for $f_\sharp$, and so we are reduced to prove the other implication.
First we observe that the specification property implies the topologically mixing one.
By~\cite{SigBauer}, $f$ is topologically mixing if and only if
$f_\sharp$ is topologically mixing. Moreover, Blokh~\cite{BL} proved that any continuous topologically
mixing interval map satisfies the specification property, thus these are equivalent properties for continuous interval maps. This proves the corollary.
\end{proof}

It is not clear to us if \cite{BL} can be extended to group actions, and so the previous equivalence
does not have immediate counterpart for group actions of continuous interval maps.

\subsection{Orbital specification properties}\label{section3}

In this subsection we introduce weaker notions of specification. In opposition to
the notion introduced in Definition~\ref{def:g-specification}, which takes into account the existence of
a metric in the group, the following orbital specification properties are most suitable for semigroups actions.
A first problem to define orbital specification properties is that group elements $g\in G$ may
have different representations as concatenation of the generators. For that reason one should explicitly mention
what is the `path', or concatenation of elements, that one is interested in tracing.

\begin{definition}\label{def:orbit-spec1}
We say that the continuous semigroup action $T: G\times X \to X$ associated to the finitely generated
semigroup $G$ satisfies the \emph{strong orbital specification property}
if for any $\vep>0$ there exists $p(\vep)>0$ such that for
any $\underline h_{p_j}\in G^*_{p_j}$ (with $p_j \ge p(\vep)$ for $1\leq j\leq k$)
any points $x_1, \dots, x_k \in X$ and any natural
numbers $n_1, \dots, n_k$, any semigroup elements
$\underline g_{n_j, j}= g_{i_{n_j}, j} \dots g_{i_2,j} \, g_{i_1,j} \in G_{n_j}$ ($j=1\dots k$) there exists $x\in X$
so that
$
d( \underline g_{{\ell}, 1} (x) \; , \; \underline g_{{\ell}, 1} (x_1) ) < \vep
$
for every $\ell =1 \dots n_1$ and
$$
d( \; \underline g_{{\ell}, j} \, \underline  h_{p_{j-1}} \, \dots \, \underline g_{{n_2}, 2} \, \underline h_{p_1} \, \underline g_{{n_1}, 1} (x) \; , \; \underline g_{{\ell}, j} (x_j) \;) <  \vep
$$
for every $j=2 \dots k$ and $\ell =1 \dots n_j$
(here $\underline g_{\ell, j} := g_{i_{\ell}, j} \dots g_{i_1,j}$).

\end{definition}

\begin{remark}\label{rmk:Subtil}
The previous notion demands that every `long word' semigroup element $h_{p_j}$ can be used to shadow
the pieces of orbits. Here, `long word' means that the element has at least one representation
that is obtained by concatenation of a large number ($\ge p_j$) of generators, the identity not included.  In the case of finitely
generated free semigroups
the representation of every element as a concatenation of generators is unique and it makes sense to notice
that the size $|h_{p_j}|$ of an element $h_{p_j}$ is well defined and coincides with $p_j$.
However, the later property holds for group actions if and only if $X$ is a unique point,
since in the case that $G$ is a group then $id \in G_{n}$ for every $n\ge 2$.
This is one of the reasons to choose $G_n^*$ instead of $G_n$.
\end{remark}

We also introduce a weaker notion of orbital specification for semigroups
inspired by some nonuniform versions for maps.

\begin{definition}\label{def:weak}
We say that the continuous semigroup action $T: G\times X \to X$ associated to the finitely generated
semigroup $G$ satisfies the \emph{weak orbital specification property}
if for any $\vep>0$ there exists $p(\vep)>0$ so that for any $p \ge p(\vep)$, there exists a set
$\tilde G_{ p}\subset G_{p}^*$ satisfying  $\lim_{p\to\infty}\frac{\sharp \tilde G_p}{\sharp G_p^*}=1$
and for which the following holds:
for any $h_{p_j}\in \tilde G_{ p_j}$ with $p_j\ge p(\vep)$, any points $x_1, \dots, x_k \in X$, any natural numbers $n_1, \dots, n_k$ and any concatenations
$\underline g_{n_j, j}= g_{i_{n_j}, j} \dots g_{i_2,j} \, g_{i_1,j} \in G_{n_j}$
with $1\leq j \leq k$ there exists $x\in X$ so that
$
d( \underline g_{{\ell}, 1} (x) \; , \; \underline g_{{\ell}, 1} (x_1) ) < \vep
$
for every $\ell =1 \dots n_1$ and
$$
d( \; \underline g_{{\ell}, j} \, \underline  h_{p_{j-1}} \, \dots \, \underline g_{{n_2}, 2} \, \underline h_{p_1} \, \underline g_{{n_1}, 1} (x) \; , \; \underline g_{{\ell}, j} (x_j) \;) <  \vep
$$
for every $j=2 \dots k$ and $\ell =1 \dots n_j$.
\end{definition}

We emphasize that the previous definitions 
are independent of the set of generators for $G$, hence these are properties intrinsic to the semigroup.
This definition weakens the later one by allowing a set of admissible elements (whose proportion increases
among all possible semigroup elements) for the shadowing.
It is not hard to check that the later notions do not depend on the set of generators for the semigroup.
Non-uniform versions of the previous orbital specification properties can be
defined in the same spirit as \cite{Ya09,OT11,Var10, PS, Thompson}, but we shall
not need or use this fact here.
In Section~\ref{section of examples} we provide examples satisfying the orbital specification property
but not the usual specification property. The following proposition is the counterpart of
Theorem~\ref{prop:G-esp pf} for orbital specification properties.

\begin{proposition}
Let $G$ be a finitely generated group. If a continuous group action $T : G \times X \to X$ satisfies
the strong (resp. weak) orbital specification property then the push-forward group action $T_{\sharp} : G \times \cP (X) \to \cP(X)$ satisfies the strong (resp. weak) orbital specification property.
\end{proposition}

\begin{proof}
Since the proofs of the two claims in the proposition are similar we shall prove the first one with
detail and omit the other. By Lemma~\ref{lemma:approx measure}, it is enough to prove the proposition
for probabilities that lie on the set $\mathcal{M}_N(X)=\{\frac1N \sum_{\ell=1}^N \delta x_\ell : x_\ell \in X\}$,
for any $N\in\mathbb{N}$.
Observe that if $T$ satisfies the strong orbital specification property then the same property holds for
the induced action $T^{(N)}$ on the product space $X^N$.
Let $\delta>0$ and take $p(\delta)\in\mathbb{N}$ given by the strong orbital specification property
of the induced action on $X^N$.
Let $\mu_1,...,\mu_k\in\mathcal{M}_N(X)$
with $\mu_j=\frac1N \sum_{l=1}^N \delta_{x_l^j}$ and  $\underline{g}_{n_j,j} \in G_{n_j}$
($1\leq j \leq k$) be given.
If we consider $\bar x_j=(x_j^1,...,x_j^N)$,  for any $|\underline{h}_{p_j}| = p_j\geq p(\delta)$
there exists $\bar x=(x_1,...,x_N)\in X^N$ such that
$d(\underline{g}_{\ell,1}(x),\underline{g}_{\ell,1}(\bar x_1))<\delta$ for every $\ell=1,...,n_1$
and
$$
d(\underline{g}_{\ell,j}\underline{h}_{p_{j-1}} \dots
\underline{g}_{n_2,2}\underline{h}_{p_1}\underline{g}_{n_1,1}(\bar x)
,\underline{g}_{\ell,j}(\bar x_j))<\delta
$$
for every $j=2,...,k$ and $\ell=1,...,n_j$.
Let $\mu=\frac1N \sum_{l=1}^N \delta_{x_l}$.
In particular $\mu $ satisfies
$d(\underline{g}_{\ell,1} \cdot \mu \;,\; \underline{g}_{\ell,1} \cdot \mu_1)<\delta$ for every $\ell=1,...,n_1$
and
$$
d(\underline{g}_{\ell,j}\underline{h}_{p_{j-1}}... \underline{g}_{n_2,2}\underline{h}_{p_1}\underline{g}_{n_1,1}
	\cdot  \mu \;,\; \underline{g}_{\ell,j} \cdot  \mu_j)<\delta,
$$
for every $j=2,...,k$ and $\ell=1,...,n_j$, which finishes the proof of the proposition.
\end{proof}

\subsection{Specification and hyperbolicity}

The relation between specification properties, uniform hyperbolicity and structural stability has been
much studied in the last decades, a concept that we will recall briefly. The content of this subsection is
of independent interest and will not be used later on along the paper.
Given a $C^1$ diffeomorphism $f$ on a compact
Riemannian manifold $M$ and an $f$-invariant compact set $\Lambda\subset M$ (that is
$f(\Lambda)=\Lambda$) we say that $\Lambda$ is \emph{uniformly hyperbolic} if there exists
a $Df$-invariant splitting $T_\Lambda M= E^s \oplus E^u$ and constants $C>0$, $0<\lambda<1$ so that
$\|Df^n(x)\mid_{E^s_x}\| \le C\lambda^n$ and $\|(Df^n(x)\mid_{E^u_x})^{-1}\| \le C\lambda^n$
for every $x\in \Lambda$ and $n\ge 1$. If $\Lambda=M$ is a hyperbolic set for $f$ then $f$ is called
an \emph{Anosov} diffeomorphism.

Originally the notion of specification was introduced by Bowen~\cite{Bowen} for uniformly hyperbolic
dynamics but fails dramatically in the complement of uniform hyperbolicity (even
partially hyperbolic dynamical systems with period points of different index do not satisfy the
specification property, see~\cite{SVY13,SVY15} for more details).
On the other hand Sakai, Sumi and Yamamoto~\cite{SSY10} proved that if the specification property
holds in a $C^1$-open set of diffeomorphisms then the dynamical systems are Anosov. It is well know that
every $C^1$ Anosov diffeomorphism $f$ is \emph{structurally stable}, that is, there exists a $C^1$-open
neighborhood $\mathcal U$ of $f$ in $\text{Diff}^1(M)$ so that for every $g\in \mathcal U$ there is
an homeomorphism $h_g: M \to M$ satisfying $g\circ h_g = h_g \circ f$.
Thus the $C^1$-robust specification implies rigidity of the underlying dynamical systems.

The previous results can be extended for finitely generated group actions acting on a compact
Riemannian manifold $M$ in a more or less direct way as we now describe.
Let $G$ be a finitely generated subgroup of $\text{Diff}^1(M)$ with generators
$G_1=\{g_1^\pm, \dots, g_k^\pm\}$.
We will say that the group action $G \times M \to M$ is \emph{structurally stable} if all the generators
are structurally stable. In other words, there are
$C^1$-neighborhoods $\mathcal U_i$ of the generators $g_i$ ($1\le i \le k$) such that for any
choice $\tilde g_i\in \mathcal U_i$ there exists a homeomorphism $h_i$ such that $\tilde g_i \circ h_i
=h_i \circ g_i$.
In the case that $G$ is abelian one can require the conjugacies to coincide (c.f. definition of structural
stability by Sad~\cite{Sad}).
We say that the group action $T: G \times M \to M$ satisfies the \emph{$C^1$-robust specification property}
if there exists a $C^1$-neighborhood $\mathcal V$ of $T$ such that any $C^1$-action $\tilde T \in \mathcal V$
satisfies the specification property. As a byproduct of the previous results we deduce the following consequence:

\begin{corollary}\label{prop:robust spec}
Let $G$ be a finitely generated subgroup of $\text{Diff}^1(M)$ such that
group action $T: G \times M \to M$ satisfies the $C^1$-robust specification property.
Then every generator is an Anosov diffeomorphism and the group action is structurally stable.
\end{corollary}

\begin{proof}
Since the group action $T: G \times M \to M$ satisfies the $C^1$-robust specification property
there exists a $C^1$-neighborhood $\mathcal V$ of $T$ such that any $C^1$-action $\tilde T \in \mathcal V$
satisfies the specification property. Moreover, from Lemma~\ref{esp of elements},
any such $\tilde T$ can be identified with a group action associated to a subgroup $\tilde G$ of  $\text{Diff}^1(M)$ whose generators $\tilde G_1=\{ \tilde g_1^\pm, \dots, \tilde g_k^\pm\}$ satisfy the specification property. This proves that the generators $g_i \in \text{Diff}^1(M)$ satisfy the $C^1$-robust specification property and,  by \cite{SSY10}, are Anosov diffeomorphisms, hence structurally stable. This proves the corollary.
\end{proof}

The previous discussion raises the question of wether the $C^1$-smoothness assumption is
necessary in the previous characterization. For instance, one can ask if a homeomorphism
satisfying the specification property $C^0$-robustly has some form of hyperbolicity.
In the remaining of this subsection we shall address some comments on this problem taking
as a simple model the push-forward dynamics, which is continuous and acts on the compact
metric space of probability measures.
Roughly, we will look for some hyperbolicity of the push-forward dynamics assuming that it
has the specification property.
Clearly, if $f$ is a topologically mixing subshift of finite type then it satisfies the specification
property and so does $f_\sharp$. On the other hand, the set of $f$-invariant measures are (non-hyperbolic)
fixed points for $f_\sharp$ and, consequently, this map does not present global hyperbolicity.
For that reason we will focus on the fixed points for the continuous
map $f_\sharp$ acting on the compact metric space $\cP(X)$.
Given $\mu\in \cP(X)$ and $\vep>0$ we define the \emph{local stable set} $W_\vep^s(\mu)$ by
$$
W_\vep^s(\mu)
	:=\{\eta \in \mathcal U : d_\cP(f_\sharp^j (\mu), f_\sharp^j (\eta)) <\vep \; \text{for every } j\ge 0\}
$$
(the \emph{local unstable set} $W_\vep^u(\mu)$ is defined analogously with $f_\sharp$ above replaced
by $f_\sharp^{-1}$).
We say that $\mu\in \cP(X)$ is a \emph{hyperbolic fixed point} for $f_\sharp$ if it is a fixed point and
there exists $\vep>0$ and constants $C>0$ and $0<\lambda<1$ so that:
\begin{itemize}
\item[(i)] $d_\cP(f_\sharp^j (\mu), f_\sharp^j (\eta)) <C \lambda^j$
	for every  $j\ge 1$ and $\eta\in W_\vep^s(\mu)$
\item[(ii)] $d_\cP(f_\sharp^{-j} (\mu), f_\sharp^{-j} (\eta)) <C \lambda^j$
		for every  $j\ge 1$ and $\eta\in W_\vep^u(\mu)$
\end{itemize}
We say that the hyperbolic fixed point is of saddle type if both stable and unstable sets are non-trivial.
Since the specification implies the topologically mixing property then we will mostly be interested in
hyperbolic fixed points of saddle type for $f_\sharp$.
It follows from the definition that hyperbolic fixed points for $f_\sharp$ are isolated. The following properties follow from the definitions and Lemma~\ref{lemma:approx measure}:
\begin{enumerate}
\item $f_\sharp$ is an affine map, that is, $f_\sharp ( t \eta + s \mu)= t f_\sharp (\eta) + s f_\sharp ( \mu)$ for 	
	every  $t,s\ge 0$ with $t+s=1$ and $\eta,\mu\in \cP(X)$
\item $\mu$ is a isolated fixed point for $f_\sharp$ if and only if the set of $f$-invariant probability measures
	satisfies $\mathcal M_f(X)=\{\mu\}$ (i.e. $f$ is uniquely ergodic),
\item $\mathcal{M}_n(X)=\{\frac{1}{n}\sum_{i=1}^n\delta_{x_i} : x_i \in X\} \subset \mathcal P(X)$
	is a closed $f_\sharp$-invariant set, and
\item $\bigcup_{n\ge 1} \mathcal M_n(X)$ is a dense subset of $\cP(X)$.
\end{enumerate}
Therefore, to analyze the existence of hyperbolic fixed points of saddle type for $f_\sharp$
that satisfies the specification property we are reduced to the case that $f$ is uniquely ergodic.
If $f$ is a contraction on a compact metric space then Banach's fixed point theorem implies the existence
of a unique fixed point that is a global attractor and, consequently, the Dirac measure at the attractor
is the unique hyperbolic (attractor) fixed point for $f_\sharp$, which is incompatible with transitivity.
However, it is nowadays well known that $C^0$-generic maps have a dense set of
periodic points (see e.g.~\cite{KMa}) and, in particular, $C^0$-generic homeomorphisms $f$ are not
uniquely ergodic. In conclusion, there is no open set of homeomorphisms $f$ so that $f_\sharp$
has a unique hyperbolic fixed point of saddle type.

\section{Specification properties and the entropy of semigroup actions}\label{sec:thermo}

The notion of entropy is one of the most important in dynamical systems, either as a topological invariant
or as a measure of the chaoticity of the dynamical system. For that reason several notions of entropy and topological pressure have been introduced for group actions in an attempt to describe its dynamical
characteristics.
As discussed in the introduction, some of the previously introduced definitions take into account
the  growth rate of the (semi)group, that is, the growth of $|G_n|$ as $n$ increases (see e.g.~\cite{Bis}
and references therein). We refer the reader to ~\cite{R.I.Grigorchuk,Harpe} for a detailed description
about growth rates for groups and geometric group theory.
In this section we characterize entropy points of semigroup actions with specification
(Theorem~\ref{thm:entropypts}) and prove that these actions have positive topological entropy
(Theorems~\ref{thm:entropy} and ~\ref{thm:entropy2}).

\subsection{Entropy points}\label{subsec:entropy}

Let $X$ be a compact metric space and $G$ be a semigroup.
First we shall introduce the notion of dynamical balls. Given $\varepsilon>0$ and $\underline g :=g_{i_{n}} \dots g_{i_2} \, g_{i_1}\in G_n$ we define
the \emph{dynamical ball} $B(x,\underline g,\varepsilon)$ by
\begin{align}\label{eq:dynball}
B(x,\underline g,\varepsilon)\nonumber
	&:= B(x,g_{i_{n}} \dots g_{i_2} \, g_{i_1},\varepsilon)\\
	&= \Big\{y\in X: d( \underline g_j (y), \underline g_j (x) ) \le \varepsilon, \; \text{ for every } 0\le j \le n \Big\}
\end{align}
where, by some abuse of notation, we set $\underline g_j:=g_{i_{j}} \dots g_{i_2} \, g_{i_1}\in G_n$ for every $1\le j \le n-1$ and $\underline g_0=id$. We also assign a metric $d_{\underline g}$ on $X$ by setting
 \begin{equation}\label{eq:dg}
  d_{\underline g} (x_1, x_2)
 	:=   d_{g_{i_{n}} \dots g_{i_2} \, g_{i_1}} (x_1, x_2)
	=\max_{0\le j \le n } \, d(\underline{g}_j(x_1),\underline{g}_j(x_2)).
 \end{equation}
It is important to notice that here both the dynamical ball and metric are adapted to the underlying
concatenation of generators $g_{i_n} \dots g_{i_1}$ instead of the group element $\underline g$,
since the later one may have distinct representations.
For notational simplicity we shall use the condensed notations $B(x,\underline g,\varepsilon)$ and
$d_{\underline g }(\cdot,\cdot)$ when no confusion is possible.
In the case that $\underline g=f^n$ the later notions coincide with the usual
notion of dynamical ball $B_f(x,n,\vep)$ and dynamical distance $d_n(\cdot, \cdot)$ with respect to
the dynamical system $f$, respectively.

Now, we recall a notion of topological entropy introduced by Ghys, Langevin, Walczak \cite{GLW}
and the notion of entropy point introduced by  Bi\'s \cite{BisII}.
Two points $x,y$ in $X$ are \emph{$(n,\varepsilon)$-separated} by $G$ if there exists
$g\in G_n$ such that $d(g(x),g(y))\geq\varepsilon$.
Given $E\subset X$, let us denote by $s(n,\varepsilon, E)$ the maximal cardinality
of $(n,\varepsilon)$-separated set in $E$. The limit
\begin{equation}\label{def:GLW}
h((G,G_1),E)=\lim_{\varepsilon\to0}\limsup_{n\to\infty}\frac{1}{n}\log s(n,\varepsilon, E)
\end{equation}
is well defined by monotonicity on $\varepsilon$. The \emph{entropy of $(G,G_1)$} is defined
by the previous expression with $E=X$. This definition depends on the generators of $G$.
In this setting of a semigroup $G$ we define by
$
B_G(x,n,\varepsilon):=\bigcap_{\underline g= g_{i_n} \dots g_{i_1} \in G_n} B(x, \underline g, \varepsilon)
$
the \emph{dynamical ball for the semigroup $G$} associated to $x$, length $n$ and size $\varepsilon$
centered at $x$, where
the intersection is over all concatenations that lead to elements in $G_n$.
This corresponds to consider points that are $\varepsilon$-close along the orbit of $x$ by
all the trajectories arising from concatenations of generators.
We say that the finitely generated semigroup $(G,G_1)$
acting on a compact metric space $X$ admits an \emph{entropy point $x_0$} if for any open
neighbourhood $U$ of $x_0$ the equality
$$
h((G,G_1),\overline{U})=h((G,G_1),X)
$$
holds. Entropy points are those for which local neighborhoods reflect the complexity of the
entire dynamical system.
In \cite{BisII}, Bi\'s proved remarkably that any finitely generated group
$(G,G_1)$ acting on a compact metric space $X$ admits an entropy point $x_0$.

In what follows we consider a semigroup action by local homeomorphisms. Recall that
for any compact metric space $X$, a continuous self map $f : X \to X$ on  is called
a \emph{local homeomorphism} if for any $x\in M$ there exists an open neighborhood $V_x$ of
$x$ so that $f\mid_{V_x} : V_x \to f(V_x)$ is an homeomorphism.
We prove that the orbital specification property for continuous semigroup actions is enough to
prove that all points are entropy points. More precisely,

\begin{theorem}\label{thm:entropypts}
Let $G \times X \to X$ be a continuous finitely generated semigroup  action on a compact Riemanian
manifold $X$  so that every element $g\in G_1$ is a local homeomorphism.
 If the semigroup action satisfies the weak orbital specification
 property then every point of X is an entropy point.
\end{theorem}
\begin{proof}
First we notice that following the proof of \cite[Theorem~2.5]{BisII} \emph{ipsis literis} we get the
existence of an entropy point  $x_0\in X$ for any finitely generated semigroup of continuous maps on $X$
(the proof does not require invertibility). Hence, for any open
neighborhood $U$ of $x_0$ it holds that  $h((G,G_1),X)=h((G,G_1),\overline U).$
Let $\zeta>0$ be arbitrary and take $\varepsilon_0=\varepsilon_0(\zeta)>0$ such that
$$
\limsup_{n\to\infty}\frac{1}{n}\log s(n,\varepsilon, \overline U)
	\ge h((G,G_1),X) -\zeta
$$
for every $0<\varepsilon \le \varepsilon_0$.

Given any $z\in X$ and $V$ any open neighborhood of $z$ we claim that
$h((G,G_1),\overline{V})=h((G,G_1),X)$.
Fix $0<\varepsilon \le \varepsilon_0$ let $p({\varepsilon})\ge 1$ be given by
the strong orbital specification property. Since there are finitely many elements in $G_{p(\varepsilon)}$,
finitely many of its concatenations and the local inverse branches of elements $\underline g : X\to X$
are uniformly continuous there exists a uniform constant $C_{\varepsilon}>0$
(that tends to zero as $\varepsilon \to 0$) so that
$\diam (\underline h^{-1}(B(y,\varepsilon))) \le C_{\varepsilon}$ for every
$\underline h \in G_{p(\varepsilon)}$ and $y\in X$.
Take $n\ge 1$ arbitrary, let $E=\{x_1,...,x_l\}\subset \overline{U}$ be a maximal
$(n,\varepsilon, \overline{U})$-separated set and consider the open set $W\subset V$
defined by the set of points $y\in V$ so that $d(y,\partial V)>C_{\varepsilon_0}$. Assume
that $0< \varepsilon \ll \varepsilon_0$ satisfies $\varepsilon + C_\varepsilon<C_{\varepsilon_0}$.

Let $\underline{g}:=g_{i_n} \dots g_{i_1} \in G_n$ be fixed.
Given a maximal $(\varepsilon, \overline{W})$-separated set $F=\{z_1,...,z_m\}$,
by the weak specification property there exists
 $\underline h=h_{i_{p(\varepsilon)}} \dots h_{i_1} \in G^*_{p(\frac{\varepsilon}{4})}$ so that
 for any $x_i\in E$ and  $z_j\in F$, there exists
 $
  y_i^j \in B( z_j, \frac{\varepsilon}4 ) \cap \underline h^{-1}( B(x_i, \underline g, \frac{\varepsilon}4 )).
 $
Since $\diam (\underline h^{-1}(B(x_i,\frac{\varepsilon}{4}))) \le C_{\frac{\varepsilon}{4}}$,
this implies that $d( \underline h^{-1}(B(x_i,\underline g, \frac{\varepsilon}{4})), \partial V ) \ge C_{\varepsilon_0}- \frac{\varepsilon}{4} - C_{\frac{\varepsilon}{4}}>0$, provided that $\varepsilon \ll \varepsilon_0$.
Thus
 $$
 \underline h^{-1} \big( B_G(x_i, n, \frac{\varepsilon}4 )\big)
 	\subset
	 \underline h^{-1} \big( B(x_i, \underline g, \frac{\varepsilon}4 )\big)
	 	\subset V
		\text{ \; for every $i$.}
 $$
By construction, the dynamical balls $(B_G(x_i, n, \frac{\varepsilon}4))_{i=1 \dots l}$
are pairwise disjoint and consequently the number of $(n+ p(\frac{\varepsilon}4), \frac{\varepsilon}4)$-separated points in $\overline V$ is at least $s(n,\varepsilon, \overline U)$.
In other words, $s \big( n+p\left(\frac{\varepsilon}{4}\right),\frac{\varepsilon}{4}, \overline{V} \big)
\ge s(n,\varepsilon, \overline{U})$
 and, consequently,
\begin{align*}
    \limsup_{n\to\infty}\frac{1}{n+p\left(\frac{\varepsilon}{4}\right)}
    		\log s \big( n+p\left(\frac{\varepsilon}{4}\right),\frac{\varepsilon}{4}, \overline{V} \big)
	& \ge \limsup_{n\to\infty} \frac{1}{n+p\left(\frac{\varepsilon}{4}\right)}\log s(n,\varepsilon, \overline{U}) \\
	& = \limsup_{n\to\infty} \frac{1}{n}\log s(n,\varepsilon, \overline{U}).
\end{align*}
The last inequalities show that $h((G,G_1),X) \ge h((G,G_1),\overline V) \ge h((G,G_1),X) -\zeta$.
Since $\zeta$ was chosen arbitrary this completes the proof of the theorem.
\end{proof}

The previous result indicates that the specification properties are powerfull tools to prove
the local complexity of semigroup actions. Observe that the previous result clearly
applies for individual transformations.

We now use the notion of topological entropy introduced in~\cite{Bufetov},
which measures the mean cardinality of separated points among possible trajectories generated
by the semigroup. Although one can expect that most finitely generated semigroups are free and so to
have exponential growth (c.f. proof of Proposition 4.5 by Ghys~\cite{Ghys} implying that for a Baire generic
set of pairs of homeomorphisms the generated group is a free group on two elements) the notion of
average entropy that we consider seems suitable for wider range of semigroups.

Let $E\subset X$ be a compact set. Given $\underline g= g_{i_n} \dots g_{i_1} \in G_n$, we say a set $K \subset E$ is \emph{$(\underline g, n, \varepsilon)$-separated set} if $d_{\underline g}(x_1,x_2) > \varepsilon$
 for any distinct $x_1,x_2\in K$. When no confusion is possible with the notation for the concatenation
 of semigroup elements, the maximum  cardinality of a
 $(\underline g,\varepsilon, n)$-separated sets of $X$ will
 be denoted by $s(\underline g, n,E, \varepsilon)$. We now recall the notion of topological entropy introduced by
Bufetov~\cite{Bufetov}.

\begin{definition}\label{def:entropyB}
Given a compact set $E\subset X$,
we define
\begin{equation}\label{eq:entropiaB}
h_{top}((G,G_1),E)
	=\lim_{\varepsilon\to 0} \limsup_{n\to\infty} \frac1n \log Z_n((G,G_1),E, \varepsilon),
\end{equation}
where
\begin{equation}\label{eq:Zn}
Z_n((G,G_1),E, \varepsilon)
	=\frac{1}{m^n}\sum_{{\underline g \in G_n^*}} s(\underline g,n, E, \varepsilon),
\end{equation}
where the sum is taken over all concatenation $\underline g$ of $n$-elements of
$G_1 \setminus \{id \}$ and $m = |G_1 \setminus \{id \}|$. The
\emph{topological entropy $h_{top}((G,G_1),X)$} is defined for $E=X$.
\\
\end{definition}

In the case that $E=X$, for simplicity reasons, we shall use simply the notations
$s(\underline g, n,\varepsilon)$ and $Z_n((G,G_1),\varepsilon)$.
It is easy to check that $h_{top}((G,G_1),X) \le h((G,G_1),X)$.
Moreover, this notion of topological entropy corresponds to the exponential growth rate
of the average cardinality of maximal separated sets by individual dynamical systems $\underline g$.
This average is taken over elements that are, roughly, in the ``ball of radius $n$ in the
semigroup $G$", corresponding
to $G_n$.
Notice that for any finite semigroup $G$, every element $g\in G$ has finite order.  In this special case, we notice that every continuous map in the generated semigroup action has zero topological entropy, which is also coherent with the definition of entropy presented in \eqref{def:GLW}.

In this context, and similarly to before, we say that $x\in X$ is an \emph{entropy point} if for
any neighborhood $U$ of $x$ one has $h_{top}((G,G_1),\overline{U}) = h_{top}((G,G_1),X)$. Our next
theorem asserts that, under the (crucial) strong orbital specification property all points are also entropy
points for this notion of entropy. More precisely,

\begin{theorem}\label{thm:entropypts2}
Let $G \times X \to X$ be a continuous finitely generated semigroup  action on a compact Riemanian
manifold $X$  so that every element $g\in G_1$ is a local homeomorphism.
 If the semigroup action satisfies the strong orbital specification then every point is an entropy point.
\end{theorem}

\begin{proof}
Given any point $z\in X$ and $V$ any open neighborhood of $z$ we claim that
$h_{top}((G,G_1),\overline{V})=h_{top}((G,G_1),X)$.
Let $\zeta>0$ be arbitrary and take $\varepsilon_0=\varepsilon_0(\zeta)>0$ such that
$$
\limsup_{n\to\infty}\frac{1}{n}\log s(n,\varepsilon, \overline U)
	\ge h((G,G_1),X) -\zeta
$$
for every $0<\varepsilon \le \varepsilon_0$. Let $p({\varepsilon})\ge 1$ be given by
the strong orbital specification property. Since there are finitely many elements in $G_{p(\varepsilon)}$,
finitely many of its concatenations and the local inverse branches of elements $\underline g : X\to X$
are uniformly continuous there exists a uniform constant $C_{\varepsilon}>0$
(that tends to zero as $\varepsilon \to 0$) so that
$\diam (\underline h^{-1}(B(y,\varepsilon))) \le C_{\varepsilon}$ for every
$\underline h \in G_{p(\varepsilon)}$ and $y\in X$.

Fix $\underline h=h_{i_{p(\varepsilon)}} \dots h_{i_1} \in G^*_{p(\frac{\varepsilon}{4})}$.
Take $n\ge 1$ and $\underline{g}:=g_{i_n} \dots g_{i_1} \in G_n$ arbitrary, let $E=\{x_1,...,x_l\}\subset X$ be a maximal
$(\underline g,n,\varepsilon)$-separated set and consider the open set $W\subset V$
defined by the set of points $y\in V$ so that $d(y,\partial V)>C_{\varepsilon_0}$.
Given a maximal $(\varepsilon, \overline{W})$-separated set $F=\{z_1,...,z_m\}$,
by the specification property, for any $x_i\in E$ and  $z_j\in F$ there exists
 $$
  y_i^j \in B( z_j, \frac{\varepsilon}4 ) \cap \underline h^{-1}( B(x_i, \underline g, \frac{\varepsilon}4 )).
 $$
 Similarly as before we deduce that
 $
 \underline h^{-1} \big( B(x_i, \underline g, \frac{\varepsilon}4 )\big)
	 	\subset V
		\text{ \; for every $i$. }
$
By construction, the dynamical balls $(B(x_i, \underline g, \frac{\varepsilon}4))_{i=1 \dots l}$
are pairwise disjoint and the points $y_i^j$ are $(\underline g \, \underline h, \frac{\varepsilon}4,\overline{V})$-separated.
This proves that
$$
s \big( \underline g \, \underline h,\frac{\varepsilon}{4}, \overline{V} \big)
	\ge s(\underline g, n, X, \varepsilon) \; s(id, 0, \overline{V}, \varepsilon)
	\geq  s(\underline g, n, X, \varepsilon).
$$
Since the elements $\underline g$ and $\underline h$ were chosen arbitrary then, summing over all possible concatenations, we deduce
\begin{align*}
    \limsup_{n\to\infty}\frac{1}{n+p\left(\frac{\varepsilon}{4}\right)}
    		& \log\Big[\frac{1}{m^{n+p\left(\frac{\varepsilon}{4}\right)}}
			\sum_{{\underline g \in G_n^*}+p\left(\frac{\varepsilon}{4}\right)}
			s \big( \underline g, n+p\left(\frac{\varepsilon}{4}\right), \overline{V}, \frac{\varepsilon}{4}\big)\Big]
			\\ %
	& \ge \limsup_{n\to\infty} \frac{1}{n+p\left(\frac{\varepsilon}{4}\right)}
		\log\Big(\frac{1}{m^{n+p\left(\frac{\varepsilon}{4}\right)}}
			\sum_{{\underline g \in G_n^*}} s \big( \underline g, n, X, \varepsilon \big)\Big)\\
  	  &= \limsup_{n\to\infty} \frac{1}{n}\log
	  \Big(\frac{1}{m^{n}}\sum_{{\underline g \in G_n^*}}   s \big( \underline g, n, X, \varepsilon \big) \Big).
\end{align*}
The last inequalities show that $h_{top}((G,G_1),X) \ge h_{top}((G,G_1),\overline V) \ge h_{top}((G,G_1),X) -\zeta$. Since both $z\in X$ and $\zeta>0$ were chosen arbitrary this completes the proof of the theorem.
\end{proof}

\subsection{Positive topological entropy}

 We now prove that orbital specification
properties are enough to guarantee that the semigroup action has positive topological entropy.

\begin{theorem}\label{thm:entropy}
Let $G$ be a finitely generated semigroup with set of generators $G_1$ and assume that
$G\times X \to X$ is a continuous semigroup action on a compact metric space $X$.
If $G\times X \to X$ satisfies the strong orbital specification property then $h_{top}((G,G_1),X)>0$.
In consequence, $h((G,G_1),X)>0$.
\end{theorem}

\begin{proof}
Since the expression in the right hand side of \eqref{eq:entropiaB} is increasing as $\varepsilon \to 0$ then it is enough to prove that there exists $\varepsilon>0$ small so that
$$ \displaystyle\limsup_{n\to\infty} \frac1n \log \frac{1}{m^n} \sum_{\underline{g}\in G_n^*}s(\underline{g},n,\varepsilon) >0.$$
Let $\varepsilon>0$ be small and fixed so that there are at least two distinct $2\varepsilon$-separated
points $x_1,x_2\in X$. Take $p(\frac{\varepsilon}{2})\ge 1$ given by the strong orbital specification property.
Taking $\underline g_{n_1,1}=\underline g_{n_2,2}=id$ and $\underline h = h_{p(\frac\varepsilon{2})}\dots
h_2 \, h_1  \in  G^*_{p(\frac\varepsilon{2})}$ there are $x_{i,j} \in B(x_i,\frac{\varepsilon}{2})$, with $i,j\in\{1,2\}$, such that $\underline h(x_{i,j}) \in B(x_j,\frac{\varepsilon}{2})$. In particular it follows that $s(\underline h, p(\frac{\varepsilon}{2}), \varepsilon)\ge 2^2$.

By a similar argument, given $\underline g:=g_{i_n}\dots g_{i_2} g_{i_1} \in G_n$ with $n=k.p(\frac{\varepsilon}{2})$, it can be written as a concatenation of $k$ elements in $G_{p(\frac{\varepsilon}{2})}$. In other words,
$\underline g= \underline h_k \dots \underline h_1$ with $\underline h_i \in G_{p(\frac{\varepsilon}{2})}$
and repeating the previous reasoning it follows that $ s(\underline g, n, \varepsilon) \ge 2^k$.
Thus,
\begin{align*} \limsup_{n\to\infty} \frac{1}{n} \log Z_n((G,G_1),\varepsilon) & \ge
	\limsup_{k\to\infty}
	\frac1{k\, p(\frac\varepsilon{2})}
	 \log \Bigg( \frac{1}{ m^{k\, p(\frac\varepsilon{2})}} \sum_{|\underline g| = {k\, p(\frac\varepsilon{2})}}
		 s(\underline g, k\, p(\frac\varepsilon{2}), \varepsilon)\Bigg)  \\
	& \ge \frac1{p(\frac{\varepsilon}{2})} \log 2.
\end{align*}
This proves that the entropy is positive and finishes the proof of the theorem.
\end{proof}

Let us observe that in~\cite{Fish} the author obtained a lower bound for the topological entropy
of $C^1$-maps on smooth orientable manifolds. Here we require continuity of the semigroup action
and a specification property (which most likely can be weakened) for deducing that topological entropy
is strictly positive.
One could expect that the weak orbital specification property could imply the semigroup action
to have positive entropy. In fact this is the case whenever the semigroup satisfies additional conditions
on the growth rate which hold e.g. for free semigroups.

\begin{theorem}\label{thm:entropy2}
Assume that $G$ is a finitely generated semigroup and that
the continuous action $G\times X \to X$ on a compact metric space $X$ satisfies the weak
orbital specification property with
\begin{itemize}
\item[(H)]
$\displaystyle\limsup_{p\to\infty} \frac{|G_p^*
		\setminus \tilde G_p|}{m^{\gamma p}}<1$ for every $0<\gamma < 1$.
\end{itemize}
Then the semigroup action $G \times X \to X$ has positive topological entropy.
\end{theorem}

In Subsection~\ref{section of examples} we give some examples of semigroups
combining circle expanding maps and rotations that satisfies the weak orbital specification
property and for which $|G_p^* \setminus \tilde G_p|$ is finite, hence (H) holds.

\begin{proof}[Proof of Theorem~\ref{thm:entropy2}]
Given $\varepsilon>0$, let $p(\varepsilon)\ge 1$ be given by the specification property.
For any $p\ge p(\varepsilon)$ let $\tilde G_p\subset G^*_p$ be given by the weak orbital specification property.  Take $n= k p$ with $p \ge p(\frac{\varepsilon}{2})$ and assume that (H) holds.

For any $\underline g \in G^*_n$ one can write it as a concatenation of $k$ elements in $G^*_{p}$, that is,
$\underline g= \underline h_k \dots \underline h_1$ with $\underline h_i \in G^*_{p}$. If this is the case,
given $0<\gamma<1$ we will say that $\underline{g}=\underline h_k \dots \underline h_1\in G^*_n$ is
$\gamma$-acceptable if
$
\sharp\{0\leq j\leq k:\underline{h}_j\in\tilde G_p \}>\gamma k.
$
Notice that
\begin{align*}
\sharp\{\underline{g}= \underline h_k \dots \underline h_1 \in G_{kp}&:\underline{g} \mbox{ not $\gamma$-acceptable}\}\\
	                                                                   &\leq\! \displaystyle\sum_{l\geq [\gamma k]}^k
	\sharp\{\underline{g}\in G_{kp}:\sharp\{0\leq j\leq k:\underline{h}_j\in G_p\backslash\tilde G_p\}=l\}.
\end{align*}
In consequence,

\begin{align*}
\sharp\{\underline{g}= \underline h_k \dots \underline h_1 \in G^*_{kp}&:\underline{g} \mbox{ not $\gamma$-acceptable}\}\\
	                                                                   &\leq\! \displaystyle\sum_{l\geq [\gamma k]}^k
	\sharp\{\underline{g}\in G^*_{kp}:\sharp\{0\leq j\leq k:\underline{h}_j\in G^*_p\backslash\tilde G_p\}=l\}.
\end{align*}
In consequence,
\begin{align}
\frac{\sharp\{\underline{g}\in G^*_{kp}:\underline{g} \mbox{ is not $\gamma$-acceptable}\}}{m^{kp}}
&\leq
	\frac{\displaystyle\sum_{l\geq [\gamma k]}^k\left(
                                              \begin{array}{c}
                                                k \\
                                               l \\
                                              \end{array}
                                            \right)
		|G^*_p|^{k-l}|G^*_p\backslash\tilde G_p|^l}{m^{kp}} \nonumber \\
&\leq k
	\frac{\left(
        \begin{array}{c}
          k \\
          \mbox{$[\gamma k]$} \nonumber  \\
        \end{array}
      \right) m^{(1-\gamma)kp} | G^*_p\backslash\tilde G_p|^k}{m^{kp}} \\
&= k
	\left( \!\!\!
        \begin{array}{c}
          k \\
          \mbox{$[\gamma k]$} \\
        \end{array}
      \!\!\!\right)
            		\Big(\frac{|G^*_p\backslash\tilde G_p|}{m^{\gamma p}}\Big)^k
	.  \label{bacana}
\end{align}
By assumption (H), given $0<\gamma_0 <1$ let $0<\delta \ll \log 2$ be small so that
$\limsup_{p\to\infty} \frac{|G^*_p\backslash\tilde G_p|}{m^{\gamma_0 p}} <e^{-2\delta}<1$.
Then by monotonicity of the later limsup in $\gamma$, it is clear that
$$
\limsup_{p\to\infty} \frac{|G^*_p\backslash\tilde G_p|}{m^{\gamma p}} < e^{-2 \delta}<1
$$
for every $\gamma \in (\gamma_0,1)$. Up to consider larger $\gamma$ sufficiently close to $1$ so that
$k\, \left(
        \begin{array}{c}
          k \\
          \mbox{$[\gamma k]$} \\
        \end{array}
      \right) \le e^{\delta k}$
for every $k$ large. The later implies that
\begin{align*}
\frac{\sharp\{\underline{g}\in G^*_{kp}:\underline{g} \mbox{ is not $\gamma$-acceptable}\}}{m^{kp}}
&  \lesssim
	e^{\delta k}
      		\Big(\frac{|G^*_p\backslash\tilde G_p|}{m^{\gamma p}}\Big)^k
  \lesssim e^{-\delta k}
\end{align*}
which decreases exponentially fast in $k$ (provided that $p$ is large enough).
Moreover, given $p\gg p(\frac{\varepsilon}{2})$ one can proceed as in the proof of the previous theorem
and prove that $s(\underline g, kp, \varepsilon) \ge 2^{\gamma k}$ for any $\gamma$-admissible
$\underline g\in G^*_{kp}$. Consequently,
\begin{align*}
\limsup_{n\to\infty} \frac{1}{n} \log Z_n((G,G_1),\varepsilon)
	& \ge
	\limsup_{k\to\infty}
	\frac1{k\, p(\frac\varepsilon{2})}
	 \log \Big(
	 	\frac{\sharp\{\underline{g}\in G^*_{kp}:\underline{g} \mbox{ is $\gamma$-acceptable}\}}
			{ m^{k p(\frac\varepsilon{2})}}
		   2^{\gamma k}  \Big)  \\
	& \ge
		 \frac1{p(\frac{\varepsilon}{2})} \log 2^{\gamma}
	   + 	\limsup_{k\to\infty}
	\frac1{k\, p(\frac\varepsilon{2})}
	 \log \left( 1- e^{-\delta k}
		\right) \\
	& \geq \frac{\gamma}{p(\frac{\varepsilon}{2})} \log 2- \frac{\delta}{p(\frac\varepsilon{2})}
\end{align*}
which is strictly positive, by the choice of $\delta$ and $\gamma$.
This proves the theorem.
\end{proof}

\section{Thermodynamics of expansive semigroup actions with specification}
\label{sec:expanding}

In this section we study thermodynamical properties of positively expansive semigroup actions
satisfying specification and also semigroups of uniformly expanding maps.
First we prove that semigroups of expanding maps satisfy the orbital specification properties
(Theorems~\ref{thm:expanding}). Then we obtain conditions for the convergence of topological pressure
(Theorem~\ref{thm:vep-generator-entropy}).
Finally we prove a strong regularity of the topological pressure function (Theorem~\ref{thm:diff}) and
prove that topological entropy is a lower bound for the exponential growth rate of
periodic points (Theorem~\ref{thm:B}).

\subsection{Semigroup of expanding maps and specification}

Throughout this subsection we shall assume that $X$ is a compact Riemannian manifold. We say that a
$C^1$-local diffeomorphism $f: M \to M$ is an \emph{expanding map} if there are constants $C>0$ and
$0<\lambda<1$ such that $\|(Df^n(x))^{-1}\| \le C \lambda^n$ for every $n\ge 1$ and $x\in X$.
\begin{theorem}\label{thm:expanding}
Let $G_1=\{g_1, g_2, \dots ,g_k\}$ be a finite set of expanding maps and let $G$ be the generated
semigroup. Then $G$ satisfies the strong orbital specification property.
\end{theorem}
The following two lemmas will be instrumental in the proof of Theorem~\ref{thm:expanding}.
\begin{lemma}\label{le:dynamical.images}
Let $g_1, \dots ,g_k$ be $C^1$-expanding maps on the compact manifold $X$. There exists
$\vep_0>0$ so that
$
\underline g ( B(x,g,\vep)  ) = B( \underline g(x), \vep)
$
for any $0<\vep \le \vep_0$, any $x\in X$ and any
$\underline g \in G$.
\end{lemma}
\begin{proof}
Let $d_i=\deg(g_i)$ be the degree of the map $g_i$. Since $g_i$ is a local diffeomorphism
there exists $\delta>0$ (depending on $g_i$) so that for every $x\in X$ setting
$g_i^{-1} (x) = \{ x_{i,1}, \dots, x_{i,d_i} \}$
there are $d_i$ well defined inverse branches $g_{i,j}^{-1}: B(x,\delta) \to V_{x_{i,j}}$ onto an
open neighborhood of $x_{i,j}$.
Since there are finitely many maps $g_i$ there exists a uniform constant $\delta_0>0$ so that
all inverse branches for $g_{i}$ are defined in balls of radius $\delta_0$.
Furthermore, since all $g_i$ are uniformly expanding all inverse branches are $\lambda$-contracting
for some uniform $0<\lambda<1$, meaning that
$
d ( \; g_{i,j}^{-1}(y) , g_{i,j}^{-1}(z)  \;)
	\le \lambda \, d(y,z)
$
for any $x\in X$, any $y,z \in B(x,\delta_0)$ and  $i=1 \dots k$.
In particular $g_{i,j}^{-1} (B(x,\delta_0)) \subset B(x_{i,j},\delta_0)$ and so
$$
V_{x_{i,j}}
	= \{ y \in X : d(y, x_{i,j})<\delta_0 \,\&\, d(g_i(y), g_i(x_{i,j})) <\delta_0 \}
	=B_{g_i}(x_{i,j},1,\delta_0).
$$
Using this argument recursively, every $\underline g_j = g_{i_j} \dots g_{i_2}\, g_{i_1} \in G_j$
is a contraction and we get that the dynamical ball
$
B(x,\underline g,\delta)
	= \bigcap_{j=0}^n \underline g_{j}^{-1} ( B( \underline g_{j}(x), \delta) )
$
(for $0<\delta<\delta_0$)
is mapped diffeomorphically by $\underline g$ onto $B( \underline g(x), \delta)$,
proving the lemma.
\end{proof}
\begin{lemma}\label{le:uniform.exactness}
Let $g_1, \dots ,g_k$ be $C^1$-expanding maps on the compact manifold $X$. For any $\vep>0$
there exists $N=N(\vep) \in \mathbb N$ so that
$
\underline g_N ( B(x,\vep) )=X
$
for every $x\in X$ and every $\underline g_N\in G^*_N$.
\end{lemma}
\begin{proof}
There exists a uniform $0<\lambda<1$ so that all inverse branches for $g_i$ are
$\lambda$-contracting for any $i$. Fix $\delta>0$. Using the compactness of $X$ it
is enough to prove that for any $x\in X$ there exists $N\ge 1$ so that
$\underline g_N (B(x,\delta))=X$ for every $\underline g_N \in G^*_N$.
Take $N=N(\delta) \ge 1$ be large and such that $\lambda^N (1+\diam X)<\delta$.
Let $\underline g_N \in G^*_N$ be arbitrary and assume, by contradiction, that
$\underline g_N (B(x,\delta))\neq X$.
Then there exists a curve $\gamma_N$ with diameter at most $\diam X +1$
connecting the points $x$ and $y\in X \setminus \underline g_N (B(x,\delta))$.
Consider a covering of $\gamma_N$ by balls of radius $\delta$ and consider $\gamma$ the image
of $\gamma_N $ by the inverse branches, such that $\gamma$ connects $x$
 to some point $z\not\in B(x,\delta)$  so that $\underline g_N(z)=y$. Using that
 $y\not\in \underline g_N (B(x,\delta))$ one gets that $z\not\in B(x,\delta)$.
  Since $\underline g_N$ is a $\lambda^N$-contraction
then $\delta < d(x,z) \le \text{length} (\gamma) \le \lambda^N (1+\diam X) < \delta$, which is a
contradiction. Thus the lemma follows.
\end{proof}

\begin{proof}[Proof of Theorem~\ref{thm:expanding}]
The proof of the theorem follows from the previous lemmas. In fact, let $\vep>0$ be fixed and consider
$x_1, x_2, \dots, x_k \in X$, natural numbers $n_1, n_2, \dots, n_k$ and
group elements $\underline g_{{n_j},j} = g_{i_{n_j}, j} \dots g_{i_2,j} \, g_{i_1,j} \in G_{n_j}$ ($j=1\dots k$).
By Lemma \ref{le:dynamical.images}, there exists $\varepsilon_0$ such that for $\varepsilon\leq\varepsilon_0$
$$
\underline{g}_{n_j}(B(x_j,\underline{g}_{n_j},\varepsilon))=B(\underline{g}_{n_j}(x_j),\varepsilon), \;
\forall1\leq j \leq k.
$$
We may assume without loss of generality that $\delta<\varepsilon_0$.
Let $p(\delta)=N(\delta)$ be given by Lemma~\ref{le:uniform.exactness}.
Given $p_1, \dots, p_k \ge p(\vep)$,  for $\underline h_{p_j} \in G^*_{p_j}$
we have that
$
\underline{h}_{p_i}(B(\underline{g}_{n_i}(x_i),\delta))=X.
$
It implies that given $\bar{x}_k\in B(x_k,\underline{g}_{n_k},\delta)$, one has
$\bar{x}_k=\underline{h}_{p_{k-1}}(\bar{x}_{k-1})$, with $\bar{x}_{k-1}\in B(\underline{g}_{n_{k-1}}(x_{k-1},\varepsilon))$,
and then $\bar{x}_k=\underline{g}_{n_{k-1}}\underline{h}_{p_{k-1}}(\bar{x}_{k-2})$, for some
$\bar{x}_{k-2}\in B(x_{k-1},\underline{g}_{n_{k-1}},\varepsilon)$.
By induction, there exists $x\in B(x_1,\underline{g}_{n_1},\varepsilon)$, such that
$$
\underline g_{{\ell}, j} \, \underline  h_{p_{j-1}} \, \dots \, \underline g_{{n_2}, 2} \, \underline h_{p_1} \, \underline g_{{n_1}, 1} (x)
\in B( x_j,\underline g_{{\ell}, j},\varepsilon)
$$
for every $j=2 \dots k$ and $\ell =1 \dots n_j$. This completes the proof of the theorem.
\end{proof}
For completeness, let us mention that the results in this subsection hold also for general
topologically mixing distance expanding maps on compact metric spaces $(X,d)$. Recall $f$ is a
distance expanding map if there are $\delta>0$ and $0<\lambda<1$  so that $d(f(x),f(y)) \ge \lambda^{-1} d(x,y)$ for every $d(x,y)<\delta$. Our motivation to focus on smooth maps comes from the fact
free semigroups can be constructed and shown to be robust in this context
(c.f. Section~\ref{section of examples}).

\subsection{Convergence and regularity of entropy and the pressure function}

In what follows we shall introduce a notion of topological pressure.
 For notational simplicity, given $\underline g\in G_{n}$ and $U\subset X$ we will use the notation
 $
 S_{\underline g}\varphi(x)=\sum_{i=0}^{n-1} \varphi (\underline g_i (x))
 $
 and
 $
   S_{\underline g}\varphi(U)=\sup_{x\in U}S_{\underline g}\varphi(x).
 $
\begin{definition}\label{eq:pressaoB}
For any continuous observable $\varphi\in C(X)$ we define the \emph{topological pressure of $(G,G_1)$
with respect to $\varphi$} by
\begin{equation}
P_{top}((G,G_1),\varphi, X):=\lim_{\varepsilon \to 0} \limsup_{n \to \infty}
	 \frac{1}{n} \log Z_n((G,G_1),\varphi, \varepsilon),
\end{equation}
where
\begin{equation}\label{eq:defZnpre}
Z_n((G,G_1),\varphi, \varepsilon)
	= \frac{1}{m^n}
		\sum_{{\underline g \in G_n^*}} \sup_{E}
		\left\{ \sum_{x\in E} e^{\sum_{i=0}^{n-1} \varphi (\underline g_i (x))} \right\}
\end{equation}
and the supremum is taken over all sets $E=E_{\underline g, n,\varepsilon}$ that are
$(\underline g, n,\varepsilon)$-separated.
\end{definition}
Observe that in the case that $G$ has only one generator $f$ then $|G_n|=|\{f^n\}|=1$
and $P_{top}((G,G_1),\varphi)$ coincides with the classical pressure $P_{\text{top}}(f,\varphi)$.
The case that the potential is constant to zero corresponds to the notion of topological entropy
introduced in Definition~\ref{def:entropyB}.
We proceed to prove that the topological pressure of expansive semigroup actions with the specification property
can be computed as a limit.
For that purpose we provide an alternative formula to compute the topological pressure
using open covers.
Given $\varepsilon>0$, $n\in \mathbb{N}$ and  $\underline{g}\in G_n$,
we say that an open cover $\mathcal{U}$ of $X$ is an $(\underline{g},n,\varepsilon)$-cover if
any open set $U\in \mathcal{U}$ has $d_{\underline{g}}$-diameter smaller than $\varepsilon$,
where $d_{\underline{g}}$ is the metric introduced in \eqref{eq:dg}. Let $cov(\underline{g},n,\varepsilon)$
be the minimum cardinality of a $(\underline{g},n,\varepsilon)$-cover of $X$.
To obtain a characterization of the topological pressure using open covers of the space we need the
continuous potential to satisfy a regularity condition.
Given $\varepsilon>0$ and $\underline{g}:=g_{i_n} \dots g_{i_1}\in G$ we define
the \emph{variation of $S_{\underline{g}}\varphi$} in dynamical balls of radius $\varepsilon$
by
$$
Var_{\underline g}(\varphi,\varepsilon)
	=\sup_{d_{\underline g}(x,y)<\varepsilon}
	|S_{\underline{g}}\varphi(x)-S_{\underline{g}}\varphi(y)|.
$$
We say that $\varphi$ has \emph{bounded distortion property} (in dynamical balls of radius
$\varepsilon$) if there exists $C>0$ so that
\begin{equation*}
\sup_{\underline{g}\in G} \; \sup_{x\in X} Var_{\underline g}(\varphi,\varepsilon) \le C.
\end{equation*}
For short we denote by $BD(\varepsilon)$ the space of continuous potentials that have
bounded distortion in dynamical balls of radius $\varepsilon$ and we say that $\varphi$ has \emph{bounded distortion property} if there exists $\varepsilon >0$
so that $\varphi$ has bounded distortion on dynamical balls of radius $\varepsilon$.
In what follows we prove that H\"older potentials have bounded distortion
for semigroups of expanding maps.

\begin{lemma}\label{le:expandings}
Let $G$ be a finitely generated semigroup of expanding maps on a compact metric space $X$ with
generators $G_1=\{g_1, \dots, g_m\}$. Then any H\"older continuous observable $\varphi: M \to \mathbb R$
satisfies the bounded distortion property.
\end{lemma}

\begin{proof}
Let $\delta_0>0$ and $0< \lambda <1$ be chosen as in the proof of the previous lemma
and assume that $\varphi$ is $(K,\alpha)$-H\"older. Given
any $0<\varepsilon<\delta_0/2$, any $\underline g=g_{i_n} \dots g_{i_1} \in G_n$
and $x,y\in X$ with $d_{\underline g}(x,y)<\varepsilon$,
\begin{align*}
 |S_{\underline{g}}\varphi(x)-S_{\underline{g}}\varphi(y)|
 	& = | \sum_{i=0}^{n-1} \varphi (\underline g_i (x)) - \sum_{i=0}^{n-1} \varphi (\underline g_i (y)) |
	 \le  \sum_{i=0}^{n-1} | \varphi (\underline g_i (x)) - \varphi (\underline g_i (y)) |  \\
	 & \le \sum_{i=0}^{n-1} K d( \underline g_i (x), \underline g_i (y) )^\alpha
	 \le \sum_{i=0}^{n-1} K \lambda^{(n-i)\alpha} d( \underline g_n (x), \underline g_n (y) )^\alpha \\
	 & \le \frac{K}{1-\lambda^\alpha} \varepsilon^\alpha.
 \end{align*}
This proves the lemma.
\end{proof}

\begin{proposition}\label{le:equivalenceP}
Let $\varphi:X\to \mathbb R$ be a continuous map satisfying the bounded distortion condition.
Then the topological pressure $P_{top}((G,G_1),\varphi, X)$ with respect to the potential $\varphi$ satisfies
 $$
P_{top}((G,G_1),\varphi, X)
	=\lim_{\varepsilon\to 0}\limsup_{n\to\infty}\frac{1}{n}\log\left(\frac{1}{m^n}\sum_{{\underline g \in G_n^*}}	
	\inf_{\mathcal{U}}\sum_{U\in\mathcal{U}}e^{S_{\underline g}\varphi(U)}\right),
 $$
 where the infimum is taken over all open covers $\mathcal{U}$ of $X$ such that
$\mathcal{U}$ is a $(\underline g,n,\varepsilon)$-open cover.
\end{proposition}

\begin{proof}
Although the proof of this proposition follows a classical argument we include it here for completeness.
Take $\varepsilon>0$, $n\in \mathbb N$ and
$\underline g\in G_n$. To simplify the notation we denote
$$
C_n((G,G_1),\varphi,\varepsilon)
	=\frac{1}{m^n}
\sum_{{\underline g \in G_n^*}}\inf_{\mathcal{U}}
\sum_{U\in\mathcal{U}}e^{S_{\underline g}\varphi(U)}
$$
where the infimum are taken over all $(\underline g,n,\varepsilon)$-open covers
and let $Z_n((G,G_1),\varphi,\varepsilon)$ be given by equation~\eqref{eq:defZnpre}.
Given a $(\underline g,n,\varepsilon)$-maximal separated set $E$ it follows that
 $\mathcal{U}= \{B(x,\underline g,\varepsilon)\}_{x\in E}$ is a
 $(\underline g,n,2\varepsilon)$-open cover. By the bounded distortion assumption,
 $
 S_{\underline g}\varphi (B(x,\underline g,\varepsilon))
 	= \sup_{z\in B(x,\underline g,\varepsilon)}S_{\underline g}\varphi(z)
	 \leq S_{\underline g}\varphi(x)+C
$
for some constant $C>0$, depending only on $\varepsilon$. Consequently,

 \begin{align}\label{eq:bd}
 \limsup_{n\to\infty}\frac{1}{n}\log C_n((G,G_1),\varphi, 2\varepsilon)
 \leq \limsup_{n\to\infty}\frac{1}{n}\log Z_n((G,G_1),\varphi,\varepsilon).
 \end{align}
On the other hand, if $\mathcal{U}$ is $(\underline g,n,\varepsilon)$-open cover,
  for any $(\underline g,n,\varepsilon)$-separated set $E\subset X$ we have that
  $\sharp E\leq \sharp \mathcal{U}$, since the diameter of any $U\in\mathcal{U}$
 in the metric  $d_{\underline g}$ is less than $\varepsilon$. By the bounded distortion
 condition  we get that

 \begin{align}\label{eq:bd2}
 \limsup_{n\to\infty}\frac{1}{n}\log Z_n((G,G_1),\varphi,\varepsilon)
 \leq \limsup_{n\to\infty}\frac{1}{n}\log C_n((G,G_1),\varphi,\varepsilon).
 \end{align}
 Now, combining equations \eqref{eq:bd} and \eqref{eq:bd2} we get that

 \begin{align}\label{eq:equalities}
 \limsup_{n\to\infty}\frac{1}{n}\log Z_n((G,G_1),\varphi,\varepsilon)
 	& \leq \limsup_{n\to\infty}\frac{1}{n}\log C_n((G,G_1),\varphi,\varepsilon) \\
	& \leq \limsup_{n\to\infty}\frac{1}{n}\log Z_n((G,G_1),\varphi,\frac{\varepsilon}2) \nonumber
 \end{align}
 and then the result follows.
\end{proof}

In the next lemma we provide a condition
under which the topological pressure can be computed as a limit.

\begin{proposition}\label{le:limitt}
Let $\varphi:X\to \mathbb R$ be a continuous potential.
Given $\varepsilon>0$, the limit superior
\begin{align*}
 \limsup_{n\to\infty}
 	\frac{1}{n} \log \Big(\frac{1}{m^n}
 \sum_{{\underline g \in G_n^*}}\inf_{\mathcal{U}}\sum_{U\in\mathcal{U}}e^{S_{\underline g}\varphi (U)}\Big)
\end{align*}
is indeed a limit.
\end{proposition}

\begin{proof}
Since $\varphi$ is continuous then it is bounded from below. Assume without loss of generality that
$\varphi$ is non-negative, otherwise we just consider a translation $\varphi+C$ since it will affect the
$\limsup$ by a translation of $C$.
Given $\varepsilon>0$, recall  that the infimum is taken over all $(\underline g,n,\varepsilon)$-open covers
$\mathcal{U}$ of $X$.
For any element $\underline g =\underline h \, \underline k \in G^*_{\ell+n}$
with $\underline h\in G_{\ell},\underline k\in G_n^* $, and any
$(\underline{h},n,\varepsilon)$-cover  $\mathcal{U}$ and
$(\underline{k},\ell,\varepsilon)$-cover $\mathcal{V}$ then
$\mathcal{W}:=\underline{k}^{-1}(\mathcal{U})\vee\mathcal{V}$ is a
$(\underline{g},\ell+n,\varepsilon)$-cover, and
\begin{align*}
 \sum_{\substack{W \in \underline{k}^{-1}(\mathcal{U}) \vee\mathcal{V} \\  W= \underline{k}^{-1}(U) \cap V}}
 	e^{t S_{\underline{g}}\varphi(W)}
	 &\leq
         	\Big( \sum_{V\in \mathcal{V}} e^{t S_{\underline{k}} \varphi(V)} \Big)	
		\Big(\sum_{U\in \mathcal{U}}e^{t S_{\underline{h}}\varphi(U )} \Big)
\end{align*}

Taking the infimum over the open covers $\mathcal{U}$ and $\mathcal{V}$ we deduce that
\begin{align*}
\inf_{\mathcal{W}}
	\Big\{ \sum_{W\in \mathcal{W}} e^{t S_{\underline{g}}\varphi(W)} \Big\}
	 &\leq
	\inf_{\mathcal{V}} \Big\{  \sum_{V\in \mathcal{V}} e^{tS_{\underline{k}}
      	\varphi(V)} \Big\}
	\;
	\inf_{\mathcal{U}}
	\Big\{
	\sum_{U\in \mathcal{U}}
		e^{tS_{\underline{h}}\varphi(U)} \Big\}.
\end{align*}
where the first infimum can be taken over all $(\underline{g},m+n,\varepsilon)$-open covers $\mathcal{W}$.
Summing over every elements $\underline g =\underline h \, \underline k \in G_{\ell+n}^*$,
\begin{align*}
  \sum_{|\underline g|=\ell+n} \inf_{\mathcal{W}}
		\Big\{ \sum_{W\in \mathcal{W}} e^{t S_{\underline{g}}\varphi(W)} \Big\}
      &
      \leq
      	\Big(\sum_{|\underline k|=\ell} \inf_{\mathcal{V}}\sum_{V\in \mathcal{V}}e^{tS_{\underline{k}}
      	\varphi(V)}\Big)
	\!
	\Big(\sum_{|\underline h|=n} \inf_{\mathcal{U}}\sum_{U\in \mathcal{U}}
		e^{tS_{\underline{h}}\varphi(U)}\Big).
\end{align*}
Thus, the sequence of real numbers $(a_n)_{n\in\mathbb N}$ given by
$$
a_n=\log\Big( \sum_{{\underline g \in G_n^*}} \inf_{\mathcal{W}}
		\Big\{ \sum_{W\in \mathcal{W}} e^{tS_{\underline{g}}\varphi(W)} \Big\}\Big)
$$
is subaditive and $\displaystyle\left\{a_n\slash n\right\}_{n\in\mathbb{N}}$ is convergent.
Since the term $\frac1n \log \frac1{m^{n}}$ is clearly constant
this completes the proof of the  proposition.
\end{proof}

From the previous results, the topological pressure can be computed as the limiting complexity of
the group action as the size scale $\varepsilon$ approaches zero.
In what follows we will be mostly interested in providing conditions for the
topological pressure of group actions to be computed as a limit at a definite size scale. Let us introduce the
necessary notions. Let $X$ be a compact metric space and $G\times X \to X$ be a continuous
action associated to the finitely generated semigroup  $(G,G_1)$.

\begin{definition}\label{defexp1}
Given $\delta^*>0$, the semigroup action $G\times X \to X$ is \emph{$\delta^*$-expansive} if for
every $x,y\in X$ there exists $k\ge 1$ and $\underline g\in G_k$ such that
$d(\underline g(x),\underline g(y))>\delta^*$.
The semigroup action $G\times X \to X$ is \emph{strongly $\delta^*$-expansive}
if for any $\gamma>0$ and any $x,y\in X$ with $d(x,y)\ge \gamma$ there exists
$k\ge 1$ (depending on $\gamma$) such that $d_{\underline g}(x,y)>\delta^*$ for all
$\underline g\in G^*_k$.
\end{definition}

\begin{remark}\label{rmk:expansive}
By compactness of the phase space $X$, a continuous action is \emph{strongly $\delta^*$-expansive} satisfies
the following equivalent formulation: given $\gamma>0$ and $x,y\in X$ with $d(x,y)\ge \gamma$
there exists $k_0\ge 1$ (depending on $\gamma$) such that
$d_{\underline g}(x,y)>\delta^*$ for all  $\underline g\in G^*_k$ and $k\ge k_0$.
\end{remark}

In what follows we prove that the topological entropy of expansive semigroup actions can be computed
as the topological complexity that is observable at a definite scale. More precisely,

\begin{theorem}\label{thm:vep-generator-entropy}
Assume the continuous action of $G$
 on the compact metric space $X$ is strongly $\delta^*$-expansive.
 Then, for every continuous potential $\varphi:X\to \mathbb R$ satisfying the bounded distortion condition
 and every $0<\varepsilon<\delta^*$
$$
P(\varphi):=P_{top}((G,G_1),\varphi,X)
	=\limsup_{n\to\infty}\frac{1}{n}\log\left(\frac{1}{m^n}\sum_{{\underline g \in G_n^*}}\sup_{E}
		\sum_{x\in E }e^{S_{\underline g}\varphi(x)}\right)
$$
where the supremum  is taken over all $(\underline g,n,\varepsilon)$-separated
sets $E\subset X$.
\end{theorem}
We just observe, before the proof, that in view of the previous characterization given in
Proposition~\ref{le:equivalenceP}, the same result as above also holds if we
consider open covers instead of separated sets.
\begin{proof}[Proof of Theorem~\ref{thm:vep-generator-entropy}]
Since $X$ is compact and $\varphi:X\to \mathbb{R}$
is continuous we assume, without loss of generality, that $\varphi$
is non negative.
Fix  $\gamma$ and $\varepsilon$ with $0<\gamma<\varepsilon<\delta^*$. We
want to show that
$$
\limsup_{n \to \infty}\frac{1}{n}\log Z_n((G,G_1),\varphi,\gamma)
	\leq\limsup_{n \to \infty} \frac{1}{n} \log Z_n((G,G_1),\varphi,\varepsilon).
$$
The other inequality is clear.
By strong $\delta^*$-expansiveness and Remark~\ref{rmk:expansive}
for any two distinct points $x, y\in X$ with $d(x,y) \ge \gamma$ there exists $k_0 \ge 1$
(depending on $\gamma$) so that  $d_{\underline{g}}(x,y) \geq \delta^*>\varepsilon$
for any $\underline{g}\in G^*_k$ and $k\ge k_0$.
Take $n\ge k_0$ and $\underline g \in G^*_{n+k}$ arbitrary and write $\underline g=
\underline h_2 \underline h_1$ with $\underline h_1 \in G^*_{n}$ and $\underline h_2 \in G^*_{k}$.
Given any $(\underline{h_1},n,\gamma)$-separated  set $E$ we claim that the set $E$ is
$(\underline{g},n+k,\varepsilon)$-separated. In fact, given $x,y\in E$ there exists a decomposition
$\underline h_1=\underline h_{1,2} \, \underline h_{1,1} \in G_n^*$ so that
$d( \underline h_{1,1}(x), \underline h_{1,1}(y) ) >\gamma$.
Using that $\underline{h}_2 \, \underline{h}_{1,2} \in \bigcup_{l\ge k} G_l^*$ and Remark~\ref{rmk:expansive}
it follows that $d_{\underline{g}}(x,y) \ge d_{\underline{h}_2 \, \underline{h}_{1,2}}( \underline h_{1,1}(x), \underline h_{1,1}(y) )>\varepsilon$ proving the claim. Now, using that $\varphi$ is non-negative,
 $$
 e^{S_{\underline g}\varphi(x)}
 	= e^{S_{\underline h_2\underline h_1}\varphi(x)}
	=e^{S_{\underline h_2}\varphi(\underline h_1(x))}e^{S_{\underline h_1}\varphi(x)}
	 \geq e^{S_{\underline h_1}\varphi(x)},
 $$
which implies that
 $Z_n((G,G_1),\varphi,\gamma)
 	\leq m^k  Z_n((G,G_1),\varphi,\varepsilon)$
because
\begin{align*}
Z_n((G,G_1),\varphi,\gamma)
	\!&\! =\frac{1}{m^n}\sum_{|\underline{h}_1|=n}
	  \sup_{E}\sum_{x\in E}e^{S_{\underline h_1}\varphi(x)} \\
	  & \!\leq\!
	\frac{m^{n+k}}{m^{n}}\frac{1}{m^{n+k}}\sum_{{\underline g \in G_n^*}+k}
	\sup_{E}\sum_{x\in E}e^{S_{\underline g}\varphi(x)}
	\!=\! m^k Z_{n+k}((G,G_1),\varphi,\varepsilon).
\end{align*}
Thus it follows that
\begin{align*}
\limsup_{n\to\infty}\frac1{n}\log Z_n((G,G_1),\varphi,\gamma)
	\leq
	\limsup_{n\to\infty}\frac1{n+k}\log Z_{n+k}((G,G_1),\varphi,\varepsilon),
\end{align*}
as we wanted to prove. This completes the proof of the theorem.
\end{proof}

Some comments on our assumptions are in order.
It is clear that if some generator for the group is an expansive map then the group is itself expansive.
Clearly, expanding maps are expansive. Moreover, the semigroup $G$ generated by $G_1=\{ g_1,...,g_k\}$
that admits some expansive generator is clearly expansive.
In Lemma~\ref{le:expexp} below we prove that semigroups of expanding maps are strongly expansive semigroups.
\begin{lemma}\label{le:expexp}
Let $G$ be a finitely generated semigroup of expanding maps on a compact metric space $X$ with
generators $G_1$. Then there exists $\delta^*>0$ so that $G$ is strongly $\delta^*$-expansive.
\end{lemma}
\begin{proof}
Let $G_1=\{g_1,...,g_m\}$ be the set of generators of $G$. Following
the proof of Lemma~\ref{le:dynamical.images}
there are uniform constants $\delta_0>0$ and $0<\lambda<1$ so that
all inverse branches $g_{i,j}^{-1}$ for $g_{i}$ are defined in balls of radius $\delta_0$ and
$
d ( \; g_{i,j}^{-1}(y) , g_{i,j}^{-1}(z)  \;)
	\le \lambda \, d(y,z).
$
for any $x\in X$, any $y,z \in B(x,\delta_0)$ and  $i=1 \dots m$.
Take $\delta_*=\delta_0/2$. Given $\gamma>0$ take $k\ge 1$ (depending on $\gamma$) so that
$\lambda^k \delta^* < \gamma$. We claim that for any $x,y\in X$ with $d(x,y) \ge \gamma$ and $g\in G^*_k$
we have $d_{\underline{g}}(x,y)>\delta^*$.
 Assume, by contradiction, that there exists $\underline{g}=g_{i_k}...g_{i_1}\in G^*_k$ with
 $d(\underline g(x), \underline g(y)) \le d_{\underline{g}}(x,y)\le \delta^*$. Then
$
d( g_{i_j}...g_{i_1}(x), g_{i_j}...g_{i_1}(y) )
	\le \lambda^{k-j} d( g_{i_k}...g_{i_1}(x), g_{i_k}...g_{i_1}(y) )
$
for every $1\le j \le k$ and so $d(x,y) \le \lambda^k d(\underline g(x), \underline g(y)) <\gamma$,
which is a contradiction. This finishes the proof of the lemma.
\end{proof}
\begin{theorem}\label{thm:diff}
Let $G$ be a finitely generated semigroup with generators $G_1$.
If the semigroup action induced by $G$ on the compact metric space $X$ is strongly
$\delta^*$-expansive and the potentials $\varphi,\psi:X\to\mathbb R$ are continuous
and satisfy the bounded distortion property then
\begin{enumerate}
 	 \item $P_{\text{top}}((G,G_1), \varphi+c,X)=P_{\text{top}}((G,G_1), \varphi,X)+c$ f
	 	or every $c\in \mathbb R$
	\item $|P_{\text{top}}((G,G_1), \varphi,X) - P_{\text{top}}((G,G_1), \psi,X)| \leq\|\varphi-\psi\|$, and
	  \item the pressure function $t\mapsto P_{\text{top}}((G,G_1),t \varphi,X)$ is an uniform limit of
	   differentiable maps.
\end{enumerate}
Moreover, $t\mapsto P_{\text{top}}((G,G_1),t \varphi,X)$ is differentiable Lebesgue-almost everywhere.
\end{theorem}

\begin{proof}
We start by observing that  property (1) follows directly from the definition of the topological pressure.
By hypothesis let $\varepsilon_0>0$ be so that $\varphi,\psi\in BD(\varepsilon_0)$.
On the one hand, by Theorem \ref{thm:vep-generator-entropy} together with equation~\eqref{eq:equalities} it follows that
for any $0<\varepsilon <\delta^*$,
$$
P(\varphi):=P_{top}((G,G_1), t\varphi,X)
	=\limsup_{n\to\infty}\frac{1}{n}\log
		\left(
		\frac{1}{m^n}\sum_{{\underline g \in G_n^*}}
		\inf_{\mathcal{U}}\sum_{U\in\mathcal{U}}e^{tS_{\underline g}\varphi(U)}
		\right)
$$
where the infimum is taken over all $(\underline g, n, \varepsilon)$-open covers $\mathcal U$.
On the other hand, by Proposition~\ref{le:limitt} the right hand side above is actually a true limit.
Thus, for any $t\in \mathbb{R}$ we have that
\begin{align}\label{eqpres}
P_{\text{top}}((G,G_1),t \varphi,X)=\lim_{n\to\infty}\frac{1}{n}\log \left(\frac{1}{m^n}\sum_{{\underline g \in G_n^*}}
\inf_{\mathcal{U}}\sum_{U\in\mathcal{U}}e^{tS_{\underline g}\varphi(U)} \right),
\end{align}
where the infimum is taken over all $(\underline g, n,\varepsilon)$-covers $\mathcal U$ for
any $0<\varepsilon<\min\{\delta^*,\varepsilon_0\}$.
It means that the map $t\mapsto P_{\text{top}}((G,G_1),t \varphi,X)$
is a pointwise limit of real analytic functions.
We claim that the convergence is indeed uniform.
To prove this we will prove that the sequence of real functions $(P_n(t\varphi))_{n\ge 1}$ defined by
$$
t\mapsto P_n(t \varphi):=\frac{1}{n} \log C_n((G,G_1),t \varphi,\varepsilon)
$$
where
$$
C_n((G,G_1),t \varphi,\varepsilon)
	=\frac{1}{m^n} \sum_{{\underline g \in G_n^*}}\inf_{\mathcal{U}}
\sum_{U\in\mathcal{U}}e^{tS_{\underline g}\varphi(U)}
$$
 is equicontinuous in compact intervals, i.e., given $\varepsilon>0$ there exists $\delta>0$ such that
 if $|t_1-t_2|<\delta$ then $|P_n(t_1\varphi)-P_n(t_2\varphi)|<\varepsilon$, for every $n\in\mathbb{N}$.
 Let $\varepsilon>0$ be fixed and take $0<\delta<\varepsilon\slash \|\varphi\|$. Given $t_1,t_2$ arbitrary with
 $|t_1-t_2|<\delta$ it holds that
 \begin{align*}
 \displaystyle|P_n(t_1\varphi)-P_n(t_2\varphi)|&=
 \frac{1}{n}\log \Bigg[\frac{\sum_{{\underline g \in G_n^*}} \inf_{\mathcal{U}}
		\left\{ \sum_{U\in\mathcal{U}} e^{t_2 S_{\underline g} \varphi (U)} \right\}}
      {\sum_{{\underline g \in G_n^*}}  \inf_{\mathcal{U}}
		\left\{ \sum_{U\in\mathcal{U}} e^{t_1 S_{\underline g} \varphi (U)} \right\}}\Bigg]
                                             \\&\leq
\frac{1}{n}\log\left[\frac{e^{n\delta\|\varphi\|}\sum_{{\underline g \in G_n^*}}  \inf_{\mathcal{U}}
		\left\{ \sum_{U\in\mathcal{U}} e^{t_1 S_{\underline g} \varphi (U)} \right\}}{\sum_{{\underline g \in G_n^*}}\inf_{\mathcal{U}}
		\left\{ \sum_{U\in\mathcal{U}} e^{t_1 S_{\underline g} \varphi (U)} \right\}}\right]
\\&=\delta\|\varphi\| <\varepsilon.
 \end{align*}
 Hence the sequence is equicontinuous. Since $(P_n(t\varphi))_{n\in \mathbb{N}}$ converges
 pointwise, we have that the sequence converges uniformly on compact intervals and so
 $t\mapsto P_{\text{top}}((G,G_1),t \varphi,X)$ is a continuous function.
Furthermore, for any $n \in \mathbb N$ the function $t\mapsto P_n(\varphi+t\psi)$ is differentiable and
 \begin{align*}
 \left|\frac{dP_n(\varphi+t\psi)}{dt}\right|
 	\!=\!\frac{1}{C_n((G,G_1),t \varphi,\varepsilon) }
	\frac{1}{n}
 \Big(\frac{1}{m^n}\sum_{{\underline g \in G_n^*}}  \displaystyle\inf_{\mathcal{U}}
		\Big\{ \sum_{U\in\mathcal{U}}S_{\underline g}\psi(U) e^{S_{\underline g} (\varphi+t\psi) (U)} \Big\}\Big)
 \end{align*}
is bounded from above by $\|\psi\|$
(here the infimum is taken over all $(\underline g, n,\varepsilon)$-covers $\mathcal U$
as in \eqref{eqpres}).
This proves property (3). Moreover, by the mean value inequality
 $$
 |P_n(\varphi)-P_n(\psi)|\leq \sup_{0\leq t\leq 1}\left|\frac{dP_n(\varphi+t(\psi-\varphi))}{dt}\right|
 \leq\|\varphi-\psi\|.
 $$
Taking $n\to\infty$ we get that
 $
 |P_{\text{top}}((G,G_1), \varphi,X)-P_{\text{top}}((G,G_1), \psi,X)|\leq\|\varphi-\psi\|
 $
 and so the pressure function $P_{\text{top}}((G,G_1), \cdot,X)$ acting on the space of
 potentials with bounded distortion is Lipschitz continuous with Lipschitz constant equal to one.
This proves property (2).
The later implies that $t\mapsto P_{\text{top}}((G,G_1), t\varphi,X)$ is Lebesgue-almost everywhere
differentiable, which concludes the proof of the theorem.
\end{proof}

\subsection{Topological entropy and growth rate of periodic points}

In the remaining of this section we prove that the topological entropy is a lower bound for the
exponential growth rate of periodic points for semigroup of expanding maps. Clearly the theorems of the
previous section apply to the topological entropy since it corresponds to the constant to zero potential.

\begin{theorem}\label{thm:B}
Let $G$ be the semigroup generated by a set $G_1=\{g_1,\dots,g_k\}$ of uniformly
expanding maps on a Riemannian manifold $X$.
Then:
\begin{itemize}
\item[(a)] $G$ satisfies the periodic orbital specification property,
\item[(b)] periodic points $Per(G)$ are dense in $X$, and
\item[(c)] the mean growth of periodic points is bounded from below as
$$
0<h_{top}((G,G_1),X)
	\leq \limsup_{n\to\infty}
	\frac{1}{n}\log \Big(  \frac{1}{m^n} \sum_{{\underline g \in G_n^*}} \sharp \text{Fix}(\underline g) \Big).
$$
\end{itemize}
\end{theorem}

\begin{proof}
Take $n\ge 1$ arbitrary and fixed.
It follows from Lemmas~\ref{le:dynamical.images} and ~\ref{le:uniform.exactness} that
there exists $\vep_0>0$ satisfying: for any $0<\vep \le \vep_0$ there exists a uniform
$N(\vep)\ge 1$ so that for any $x\in X$, any $\underline g_n \in G_n$ and $\underline g_N\in G^*_N$
with $N\ge N(\vep)$ it holds
$$
\underline g_N ( \underline g_n ( B(x,g,\vep)  ) ) = X.
$$
Consider $\vep>0$, $x_1, x_2, \dots, x_k \in X$, natural numbers $n_1, n_2, \dots, n_k$ and
group elements $\underline g_{{n_j},j} = g_{i_{n_j}, j} \dots g_{i_2,j} \, g_{i_1,j} \in G_{n_j}$ ($j=1\dots k$) be given and let us prove that $G$ satisfies the periodic orbital specification property, that is, there
exists a periodic orbit shadowing the previously defined pieces of orbit.
For that let us define $x_{k+1}=x_1$ and $\underline g_{n_{k+1}}=\underline g_1 \in G_{n_1}$.

By the proof of Theorem~\ref{thm:expanding},  there exists $p(\delta)\ge 1$ so that for any $p_1, \dots, p_k \ge p(\vep)$,  for $\underline h_{p_j} \in G^*_{p_j}$
we have that
$
\underline{h}_{p_i}(B(\underline{g}_{n_i}(x_i),\delta))=X.
$
Hence, there is a well defined inverse branch (which we denote by $\underline g_{n_i}^{-1}
\underline h_{p_i}^{-1}$ for simplicity) so that
$$
\underline g_{n_i}^{-1}  \underline h_{p_i}^{-1} ( B(x_{i+1},g_{n_{i+1}},\vep) )
	\subset  B(x_i,g_{n_i},\vep)
$$
 and  $\underline g_{n_i}^{-1}  \underline h_{p_i}^{-1} \mid_{B(x_{i+1},g,\vep)}$ is a
contraction. Since, $B(x_{k+1},g_{n_{k+1}},\vep) = B(x_{1},g_{n_{1}},\vep)$,
$$
\underline g_{n_1}^{-1}  \underline h_{p_{1}}^{-1} \dots \underline g_{n_k}^{-1}  \underline h_{p_{k}}^{-1} ( B(x_{k+1},g_{n_{k+1}},\vep) )
	\subset B(x_{1},g_{n_{1}},\vep)
$$
and the composition $\underline g_{n_1}^{-1}  \underline h_{p_{1}}^{-1} \dots \underline g_{n_k}^{-1}  \underline h_{p_{k}}^{-1}$ is a uniform contraction, then there exists a unique repelling fixed point for
$\underline h_{p_{k}} \underline g_{n_k} \dots  \underline h_{p_{1}} \underline g_{n_1}$ in
the dynamical ball $B(x_{1},g_{n_{1}},\vep)$. By construction, the fixed point for
$\underline h_{p_{k}} \underline g_{n_k} \dots  \underline h_{p_{1}} \underline g_{n_1}$
shadows the specified pieces of orbits.
This proves that $G$ satisfies the periodic orbital specification property in (a). Clearly (b) is a consequence
of the first claim (a).

Now, take $\underline{g}\in G^*_n$ and observe that for any maximal
$(\underline{g},n,2\vep)$-separated set $E$, the dynamical balls $\{ B(x, \underline{g}, \vep) : x\in E \}$ form a pairwise disjoint collection. Let $p(\vep)$
be given by the previous periodic orbital specification property.
For any arbitrary $\underline k \in G^*_{n+p(\vep)}$ one can write
$\underline k = \underline h_g \; \underline g$ for $\underline g\in G_{n}^*$ and
$\underline h_g\in G^*_{p(\vep)}$.
Notice that, proceeding as before,
$$
\underline k ( B(x, \underline g, \delta))
	= \underline{h}_g (B(\underline{g} (x),\delta))
	=X
$$
for every $x\in E$ and so there is a unique fixed point for $\underline k$ on the dynamical ball
$B(x, \underline g, \delta)$. This yields $\text{Fix} (\underline k) \ge \sharp E$
and so

\begin{align*}
\sum_{|\underline{k}|={n+p(\vep)}} \sharp\text{Fix} (\underline k)
	 \ge \sum_{|\underline{g}| =n}  \sharp\text{Fix} (\underline{h}_g \, \underline g)
	 \ge \sum_{|\underline{g}| =n} s(\underline{g}, n, 2\delta).
\end{align*}
Therefore,
\begin{align*}
 \limsup_{n\to\infty}\frac{1}{n}\log \Big(\frac1{m^n} \sum_{|\underline{k}|=n} \sharp\text{Fix} (\underline k) \Big)
	& =
	 \limsup_{n\to\infty}\frac{1}{n}\log \Big(\frac1{m^{n+p(\vep)}}\sum_{|\underline{k}|={n+p(\vep)}} \sharp	
	 \text{Fix} (\underline k) \Big) \\
 	& =  \limsup_{n\to\infty}\frac{1}{n}\log \Big(\frac1{m^n} \sum_{|\underline{k}|={n+p(\vep)}} \sharp\text{Fix} (\underline k) \Big)  \\
 &
 \ge \limsup_{n\to\infty}\frac{1}{n}\log\Big( \frac1{m^n} \sum_{|\underline{g}|=n} s(\underline{g}, n, 2\delta)\Big) .
 \end{align*}
Taking $\vep\to 0$ in the left hand side the previous inequality and recalling Theorem~\ref{thm:entropy}
this proves (c) and finishes the proof of the theorem.
\end{proof}

Some comments are in order. Firstly it is not hard to check that an analogous result holds for the notion
of entropy $h((G,G_1),X)$, leading to
$$
h((G,G_1),X) \leq \limsup_{n\to\infty}\frac{1}{n}\log\sharp Per(G_n).
$$
Secondly, since any expanding map satisfies the periodic specification property then
periodic measures are dense in the space of invariant probability measures
(see e.g. \cite[Proposition~21.8]{DGS}). Hence, given a finitely generated semigroup of
expanding maps $G$ it is clear that whenever the set $\mathcal M(G)$ of probability
measures invariant by every element $g\in G$ is non-empty then
the set of periodic measures
$$
\cP_{per}(G)
	= \bigcup_{n\ge 1} \bigcup_{\underline g\in G_n}
		\Big\{ \frac1n \sum_{j=0}^{n-1}  \delta_{\underline g_j (x)}:  x\in Fix(\underline g) \Big\}
$$
is dense in the set of probability measures $\mathcal{M}(G)$.
Finally, weighted versions of the previous theorem for potentials with bounded distortion
are also very likely to hold.

\section{Applications}\label{section of examples}

In this section we provide some classes of examples of semigroup actions that combine hyperbolicity and specification properties. We also provide some examples for which while we compare the notions
of topological entropy used here with some others previously introduced and available in the literature,
and discuss the relation between entropy, periodic points and specification properties.
\medskip

The following example illustrates that in the notion of specification some `linear independence condition'
on the set of generators must be assumed in order to obtain that the group has the specification property.

\begin{example}\label{ex:A}
Consider the integer valued matrix
\begin{equation}\label{eq:hyperbol}
A=\left(
                                     \begin{array}{cc}
                                       2 & 1 \\
                                       1 & 1 \\
                                     \end{array}
                                   \right),
\end{equation}
which induces a linear (topologically mixing) Anosov $f_A$ on $\mathbb{T}^2
=\mathbb R^2 / \mathbb Z^2$ that satisfies the specification property. Hence, the
$\mathbb Z$ action $\mathbb Z \times \mathbb{T}^2 \to \mathbb{T}^2$
given by $(n,x)\mapsto f_A^n(x)$ satisfies the specification property.

Now, take $B=A^{-2} \in SL(2,\mathbb Z)$ which  also induces a linear Anosov $f_B$
on the torus and satisfies the specification property.  Nevertheless, the $\mathbb Z^2$-action
$\mathbb Z^2\times \mathbb T^2 \to \mathbb T^2$ given by
$((m,n),x)\mapsto f_A^m( f_B^n(x))=f_A^{m-2n}(x)$
clearly does not satisfy the specification property
because every element in the (unbounded) subgroup
$\{(2n,n) : n\in\mathbb Z\} \subset \mathbb Z^2$ induces the identity map.
 This indicates that generators should be taken
in an irreducible way, that is, that there are $n_1,...,n_k\in \mathbb{Z}$ not all simultaneously zero
so that $g_1^{n_1}...g_k^{n_k}=Id_G$.
\end{example}

The next modification of the previous example illustrates that
the irreducibility of the generators in the sense that two generators $A$ and $B$ satisfy
$A^mB^n\not=Id$ for all $m,n\in\mathbb{Z}$ not simultaneously zero is not the unique obstruction.

\begin{example}\label{Subtil}
Let $A, B$ be the two matrices in $SL(4,\mathbb{Z})$ given by
$$
A=\left(
         \begin{array}{cc}
           \mathcal{A} & 0 \\
           \mathcal{I}_2 &  \mathcal{A}\\
         \end{array}
       \right)
\quad\text{and}\quad
B=\left(
         \begin{array}{cc}
           \mathcal{A} & 0 \\
           0 &  \mathcal{A}\\
         \end{array}
       \right)
, \quad \text{where} \quad
\mathcal{A}=\left(
         \begin{array}{cc}
           2 & 1 \\
           1 &  1\\
         \end{array}
       \right)
       \in SL(2,\mathbb{Z}),
$$
$\mathcal{I}_2\in \mathcal{M}_{2 \times 2}(\mathbb{Z})$  denotes the identity matrix and $0 \in \mathcal{M}_{2 \times 2}(\mathbb{Z})$ is the null matrix. It is not difficult to see that $A$ and $B$ are hyperbolic matrices
(hence the diffeomorphisms induced by $A$ and $B$ satisfy the specification property), these commute
but $B\not=A^m$ for all $m\in \mathbb{Z}$.
Consider the $\mathbb Z^2$-action $T:\mathbb Z^2 \times \mathbb{T}^4\to\mathbb{T}^4$
of $\mathbb Z^2$ on the torus $\mathbb{T}^4$ defined by
$((m,n),x) \mapsto A^m B^n(x)$.
Since the element
$$
A^{-1}B=\left(
         \begin{array}{cc}
           \mathcal{I}_2 & 0 \\
           \mathcal{I}_2 &  \mathcal{I}_2\\
         \end{array}
       \right)
$$
does not satisfy the specification property one can deduce from
Lemma~\ref{esp of elements} that this group action does not satisfy the specification property.
Similarly, it is not hard to check that
this group action does not satisfy neither of the orbital specification properties.
\end{example}

It follows from the discussion on the previous section that $C^1$-robust specification property implies that
the corresponding generators are uniformly hyperbolic and, in particular, the action is structurally stable.
Our twofold purpose in the next example is: (i) to exhibit broad families of  non-hyperbolic
smooth maps that satisfy orbital specification properties although generators do not necessarily have the
specification property;
(ii) present examples where the weak orbital specification property holds while the strong orbital property
does not.

\begin{example}\label{dxx}
 Let $f:\mathbb{S}^1\to \mathbb{S}^1$ be  a $C^1$-expanding map
 of the circle and $R_\alpha:\mathbb{S}^1\to \mathbb{S}^1$
 be the rotation of angle $\alpha$. Let $G$ be the semigroup
 generated by $G_1=\{id, f,R_\alpha\}$.
This example can be modified for the semigroup $G$ to be free (e.g. by taking a irrational rotation
and an expanding map with trivial centralizer c.f. discussion in the Example~\ref{ex:central}).

 \vspace{.15cm}
\noindent  \textit{ Claim 1:} The action induced by the semigroup $G$ on the unit circle $\mathbb{S}^1$ does not
 satisfy the strong orbital specification property.

\begin{proof}[Proof of Claim~1]
Take $\vep>0$ and $x_1,\neq x_2$ in the circle, $n_1=n_2=n \ge 1$ and the maps
$\underline g_{n_1}=f^{n_1}$ and $\underline g_{n_2}=f^{n_2}$. For any $p\ge 1$ take
$\underline h_p=R_\alpha^p=R_{\alpha p}$ the rotation of angle $\alpha p$.
If $n$ is large then the dynamical balls $B_f(x_1,n_1,\vep)$ and $ B_f(x_2,n_2,\vep)$ are disjoint and small.
In particular, there exists $p\ge 1$ so that $\underline h_p(B_f(x_1,n_1,\vep)) \cap B_f(x_2,n_2,\vep)
=\emptyset$. In particular the semigoup action $G$ on $\mathbb S^1$ does not satisfy the strong
specification orbital property.
\end{proof}

 \vspace{.15cm}
\noindent \textit{Claim 2:} The action induced by the semigroup $G$ on the unit circle $\mathbb{S}^1$ satisfies the weak orbital specification property.

\begin{proof}[Proof of Claim~2]
Since $f$ is $C^1$-expanding, by the proof of Lemma \ref{le:dynamical.images},
there exists $\vep_0>0$ so that for any $0<\vep \le \vep_0$, any $x\in X$ and any $ n\in \mathbb{N}$ it follows
that $f^n ( B_f(x,n,\vep)  ) = B(  f^n(x), \vep ).$ Moreover, there exists $N=N(\vep)\ge 1$ so that any ball
of radius $\vep$ is mapped onto $\mathbb S^1$ by $f^N$.
We can now prove the claim. Given $\vep>0$ take $p(\vep)=N(\vep)\ge 1$.
For any $p\ge p(\vep)$ let $\tilde G_{ p}\subset G_{p}^*$ denote the set of elements
$\underline h_{p} \in G_{p}^*$ for which the following holds:
given arbitrary points $x_1, \dots, x_k \in X$, any positive integers
$n_1, \dots, n_k \ge 1$, any elements
$\underline g_{n_j, j}= g_{i_{n_j}, j} \dots g_{i_2,j} \, g_{i_1,j} \in G_{n_j}$
and any elements $h_{p_j}\in \tilde G_{p_j}$ with $p_j \ge p(\delta)$
there exists $x\in X$ so that
$
d( \underline g_{{\ell}, 1} (x) \; , \; \underline g_{{\ell}, 1} (x_1) ) < \vep
$
for every $\ell =1 \dots n_1$ and
$$
d( \; \underline g_{{\ell}, j} \, \underline  h_{p_{j-1}} \, \dots \, \underline g_{{n_2}, 2} \, \underline h_{p_1} \, \underline g_{{n_1}, 1} (x) \; , \; \underline g_{{\ell}, j} (x_j) \;) <  \vep
$$
for every $j=2 \dots k$ and $\ell =1 \dots n_j$.
We claim that
$
\lim_{p\to+\infty} |\tilde G_p| / |G_p^*|=1.
$
We notice that
$
\underline g_{n_j, j} ( B(x, \underline g_{n_j, j}, \vep) ) = B(\underline g_{n_j, j}(x),\vep)
$
is a ball of radius $\vep$ for any $1\le j \le k$. So, if the expanding map is $f$ is combined at least $p(\vep)$
times in any way in the words $\underline h_p$ we get $\underline h_p(B(y,\vep))=\mathbb S^1$ for any
$y$ which clearly implies that $\underline h_p \in \tilde G_p$. Thus for any $p\ge p(\delta)$
$$
G_p^* \setminus \tilde G_p
	\subset \Big\{ \underline h_p=h_{i_p} \dots h_{i_2} h_{i_1} \in G_p : \sharp\{ 1\le j \le p: h_{i_j}=f\} < p(\vep) \Big\}.
$$
Clearly, for any $0<\gamma < 1$
\begin{equation}\label{eqH3}
\frac{ |G_p^* \setminus \tilde G_p| }{ 2^{\gamma p} }
	\le 2^{-\gamma p}
		\sum_{k=0}^{p(\delta)-1}
		\Big(\begin{array}{c}
		p \\ k
		 \end{array}
		 \Big)
	 \le
		 p(\vep)  \, 2^{- \gamma p} \, p^{p(\delta)}
		 \to 0
\end{equation}
as $p$ tends to infinity, which proves our claim.
\end{proof}
\noindent
Since the assumption (H) in Theorem~\ref{thm:entropy2} is a direct consequence of the previous equation~\eqref{eqH3} then we deduce that this semigroup action has positive topological entropy.
\end{example}

Clearly we can modify the previous strategy to deal with semigroups
with more generators or non-expanding maps.
Our next result illustrates that no generator of a semigroup need to have uniform
expansion for the semigroup to have weak orbital specification.  
We illustrate this fact with the following example.
\begin{example}
For any $\beta>0$, consider the interval map $f_\beta : [0,1] \to [0,1]$ given by
$$
f_\beta(x)
	=\begin{cases}
	\begin{array}{ll}
	x ( 1+ (2x)^\beta ) &, \text{if}\; x\in [0,\frac12] \\
	2x-1  & ,\text{if}\; x\in (\frac12,1]
	\end{array}
	\end{cases}
$$
also known as Maneville-Pomeau map. 
 Although $f_\beta$ is not continuous it induces 
a continuous and topologically mixing circle map $\tilde f_\beta$ taking $\mathbb S^1=[0,1]/\sim$ with the 
identification $0\sim 1$.  
Let $G$ be the semigroup generated by $G_1=\{id, \tilde f_\beta,R_\alpha\}$ where $R_\alpha$ is 
the rotation of angle $\alpha$. Clearly no element of $G_1$ is an expanding map. We claim 
that $G$ satisfies the weak orbital specification property. First we observe that since $R_\alpha$ is an isometry
then for every $x\in \mathbb S^1$, every $n\ge 1$, every $\underline g \in G_n^*$ and $\vep>0$ the dynamical ball $B(x,\underline g,\vep)$ 
satisfies 
$\underline g (\, B(x,\underline g,\vep) \,) = B(\underline g(x), \vep)$. Second, although $\tilde f_\beta$ is not uniformly
expanding it satisfies the following scaling property: 
$
\diam(\tilde f_\beta ([0,\delta]))  
	\ge \frac\delta2 + \frac\delta2 [1+ (1+\beta) \delta^\beta] = c_\delta \diam([0,\delta])
$
and 
$
\diam(\tilde f_\beta (I)) 
	\ge \sigma_\delta \diam (I)
$
for every ball $I \subset \mathbb S^1$ of diameter larger or equal to $\delta$,
where $c_\delta:=(1+ \delta (1+\beta) \delta^\beta)>1$
(here we use $f_\beta'(x)=1+ (1+\beta)2^\beta x^\beta \ge 1+ (1+\beta) \delta^\beta$ for every $x\in [\frac\delta2,\frac12]$ 
and $f_\beta'(x)=2$ for every $x\in (\frac12,1]$). Using the previous expression recursively, we deduce that there exists $N_\vep>0$
so that 
$$
\underline g (B(x,\vep)) = \mathbb S^1
$$
for every $x\in \mathbb S^1$, and every $\underline g:= g_{i_n} \dots g_{i_1}\in G_n^*$ such that $\sharp\{1\le j \le n : g_{i_j} =\tilde f_\beta\} \ge N_\vep$. The proof of the weak orbital specification property follows as in Example~\ref{exx}.
\end{example}

Our next purpose is to provide an example of a semigroup with exponential growth that
is not a free semigroup but still satisfy the assumptions of Theorem~\ref{thm:diff}.

\begin{example}\label{ex:central}
Let $X=\mathbb S^1$ be the circle and consider the expanding maps on $S^1$ given by
$g_1(x)=2x \; (\!\!\! \mod 1)$, that $g_2(x)=3x \; (\!\!\! \mod 1)$. It is clear that these maps comute (that is,
$g_1\circ g_2=g_2\circ g_1$) and that $g_1^k \neq g_2^\ell$ for every $k, \ell \in\mathbb Z_+$ (since $2$ and
$3$ are relatively prime). Now, consider another $C^1$-expanding map $g_3$ such that its centralizer
$Z(g_3)$ is trivial, meaning $$ Z(g_3):=\{ h: \mathbb S^1 \to \mathbb S^1 \text{ expanding } \colon  h\circ g_3= g_3 \circ h \}= \{ g_3^\ell \colon \ell \in\mathbb Z_+\}.$$ In particular the subgroup generated by $g_2$ and $g_1$ is disjoint from $Z(g_3)$. In other words, $g_3 \circ g_2^\ell \circ g_1^k \neq  g_2^\ell \circ g_1^k \circ g_3$ for every $\ell, k \in \mathbb Z_+$. The existence of such $g_3$ is garanteed by \cite{Ar}. Let $G$ be the semigroup of expanding maps with generators $G_1=\{g_1, g_2, g_3\}$. By construction, the subgroup $\tilde G$ of $G$ generated by $\tilde G_1=\{g_1, g_3\}$ is a free semigroup then
$$\lim_{n\to\infty} \frac1n \log {|G_n|} \ge \log 2 >1$$ and the semigroup has exponential growth. Since the generators do not have finite order then any elements $\underline g\in G_n$ is a concatenation
$\underline g = g_{i_n} \dots g_{i_1}$ with $g_{i_j} \in G_1$.
By commutativity, all concatenations of $j$ elements $g_1$ and $k$ elements $g_2$ coincide with the
expanding map $g_1^j \, g_2^k$ and consequently there are exactly $n+1$ elements in $G_n$ obtained
as concatenations of the elements $g_1$ and $g_2$.
This semigroup has exponential growth and is not abelian but still satisfies the conditions of
Theorem~\ref{thm:diff} for every H\"older continuous potential $\varphi: X \to\mathbb R$ and, in particular,
the pressure function $t\mapsto P_{\text{top}}((G,G_1),t \varphi,X)$ is differentiable
Lebesgue-almost everywhere.
\end{example}

In what follows we shall provide a simple example of a $\mathbb Z^d$-semigroup action
where we can already discuss the relation between the notion of topological
entropy that we introduced in comparison with some of the previous ones.
We focus on the case of semigroups of expanding maps for simplicity of computations
while we notice that an example of actions of total automorphisms as considered in
Example~\ref{ex:A} could be constructed analogously.

\begin{example}\label{ex:comparacao}
Let $X=\mathbb S^1$ be the circle and the $\mathbb Z^3$-group action
$T : \mathbb Z^3 \times \mathbb S^1 \to \mathbb S^1$
defined by $((m,n,k), x) \mapsto g_1^m g_2^{n} g_3^k (x)$, where
$g_1(x)=2x \; (\!\!\! \mod 1)$, $g_2(x)=3x \; (\!\!\! \mod 1)$ and $g_3(x)=5x \; (\!\!\! \mod 1)$
are commuting expanding maps of the circle.
By commutativity and the fact that the numbers $2,3,5$ are relatively prime
it is easy to check that
$
|G_n| = (n+1) (n+2) / 2.
$
First we shall compute the topological pressure as considered by Bis in~\cite{Bis}.
If $s(n,\delta)$ denotes the number of $(n,\delta)$-separated sets by $G$
the topological entropy in \cite{Bis} is defined by
\begin{equation}\label{defBis}
\lim_{\delta\to 0} \limsup_{n\to\infty} \frac1{|G_{n-1}|} \log s(n,\delta).
\end{equation}
In our context, for any $\delta>0$
$$
\limsup_{n\to\infty} \frac1{|G_{n-1}|} \log s(n,\delta)
	\le \limsup_{n\to\infty}  \frac{2}{n^2} \log (5^n)
	= 0
$$
proving that the entropy in \eqref{defBis} is zero.
For the sake of completeness let us mention that it is  remarked in \cite{Bis} that having positive topological entropy with this definition does not depend on the generators.
Ruelle~\cite{Ruelle} considered a slightly different but similar notion of topological entropy but that does coincide with
\eqref{defBis} in this context.

Let us now proceed to compute the notion of topological entropy considered by Ghys, Langevin, Walczak
~\cite{GLW} and Bis~\cite{BisII}. According to their definition entropy is computed as
$$
\lim_{\delta\to 0} \limsup_{n\to\infty}
	\frac1n \log s(n,\delta)
	= \log 5
$$
and it measures the \emph{maximal} entropy rate in the semigroup.
Finally we observe that it follows from \cite{Bufetov2} that the topological entropy of the semigroup
action, according to Definition~\ref{def:entropyB}, in the case the generators are expanding is given by
\begin{equation*}
h_{top}((G,G_1),X)
	= \log \big( \frac{\deg g_1 + \deg g_2 + \deg g_3}{3} \big)
	= \log \big( \frac{10}{3} \big)
	>0.
\end{equation*}
Finally let us mention that this semigroup action satisfies the strong orbital specification properties
and, consequently,  it follows from Theorems~\ref{thm:entropypts} and ~\ref{thm:entropypts2}
that every point in the circle is an entropy point with respect to both entropy notions.
\end{example}

\vspace{.5cm}

\subsection*{Acknowledgements}
F.G. is supported by BREUDS and P.V. is supported by a postdoctoral fellowship by CNPq-Brazil and are grateful to Faculdade de
Ci\^encias da Universidade do Porto for the excellent research conditions.


\begin{thebibliography}{10}

\bibitem{Ar}
Carlos Arteaga.
\newblock Centralizers of expanding maps on the circle.
\newblock {\em Proc. Amer. Math. Soc.}, 114(1):263--267, 1992.

\bibitem{SigBauer}
Walter Bauer and Karl Sigmund.
\newblock Topological dynamics of transformations induced on the space of
  probability measures.
\newblock {\em Monatsh. Math.}, 79:81--92, 1975.

\bibitem{Bis}
Andrzej Bi{\'s}.
\newblock Partial variational principle for finitely generated groups of
  polynomial growth and some foliated spaces.
\newblock {\em Colloq. Math.}, 110(2):431--449, 2008.

\bibitem{BisII}
Andrzej Bi{\'s}.
\newblock An analogue of the variational principle for group and pseudogroup
  actions.
\newblock {\em Ann. Inst. Fourier (Grenoble)}, 63(3):839--863, 2013.

\bibitem{BiU}
Andrzej Bi{\'s} and Mariusz Urba{\'n}ski.
\newblock Some remarks on topological entropy of a semigroup of continuous
  maps.
\newblock {\em Cubo}, 8(2):63--71, 2006.

\bibitem{BiW}
Andrzej Bi{\'s} and Pawe{\l} Walczak.
\newblock Entropy of distal groups, pseudogroups, foliations and laminations.
\newblock {\em Ann. Polon. Math.}, 100(1):45--54, 2011.

\bibitem{BL}
A.~M. Blokh.
\newblock Decomposition of dynamical systems on an interval.
\newblock {\em Uspekhi Mat. Nauk}, 38(5(233)):179--180, 1983.

\bibitem{Bowen}
Rufus Bowen.
\newblock Periodic points and measures for {A}xiom {$A$} diffeomorphisms.
\newblock {\em Trans. Amer. Math. Soc.}, 154:377--397, 1971.

\bibitem{Bo75}
Rufus Bowen.
\newblock {\em Equilibrium states and the ergodic theory of {A}nosov
  diffeomorphisms}, volume 470 of {\em Lecture Notes in Mathematics}.
\newblock Springer-Verlag, Berlin, revised edition, 2008.
\newblock With a preface by David Ruelle, Edited by Jean-Ren{\'e} Chazottes.

\bibitem{BR75}
Rufus Bowen and David Ruelle.
\newblock The ergodic theory of {A}xiom {A} flows.
\newblock {\em Invent. Math.}, 29(3):181--202, 1975.

\bibitem{Bufetov}
A.~Bufetov.
\newblock Topological entropy of free semigroup actions and skew-product
  transformations.
\newblock {\em J. Dynam. Control Systems}, 5(1):137--143, 1999.

\bibitem{Bufetov2}
A.~I. Bufetov.
\newblock Ergodic theorems for actions of several mappings.
\newblock {\em Uspekhi Mat. Nauk}, 54(4(328)):159--160, 1999.

\bibitem{CKN}
Grant Cairns, Alla Kolganova, and Anthony Nielsen.
\newblock Topological transitivity and mixing notions for group actions.
\newblock {\em Rocky Mountain J. Math.}, 37(2):371--397, 2007.

\bibitem{ChuLi}
Nhan-Phu Chung and Hanfeng Li.
\newblock Homoclinic groups, {IE} groups, and expansive algebraic actions.
\newblock {\em Invent. Math.}, 199(3):805--858, 2015.

\bibitem{CRV}
Maria~Pires de~Carvalho, Fagner~B. Rodrigues, and Paulo Varandas.
\newblock Semigroup actions of expanding maps.
\newblock {\em preprint}.

\bibitem{Harpe}
Pierre de~la Harpe.
\newblock {\em Topics in geometric group theory}.
\newblock Chicago Lectures in Mathematics. University of Chicago Press,
  Chicago, IL, 2000.

\bibitem{DGS}
Manfred Denker, Christian Grillenberger, and Karl Sigmund.
\newblock {\em Ergodic theory on compact spaces}.
\newblock Lecture Notes in Mathematics, Vol. 527. Springer-Verlag, Berlin-New
  York, 1976.

\bibitem{EKW}
A.~Eizenberg, Y.~Kifer, and B.~Weiss.
\newblock Large deviations for {${\bf Z}^d$}-actions.
\newblock {\em Comm. Math. Phys.}, 164(3):433--454, 1994.

\bibitem{Fish}
A.~Yu. Fishkin.
\newblock An analogue of the {M}isiurewicz-{P}rzytycki theorem for some
  mappings.
\newblock {\em Uspekhi Mat. Nauk}, 56(1(337)):183--184, 2001.

\bibitem{Fried}
Shmuel Friedland.
\newblock Entropy of graphs, semigroups and groups.
\newblock In {\em Ergodic theory of {${\bf Z}^d$} actions ({W}arwick,
  1993--1994)}, volume 228 of {\em London Math. Soc. Lecture Note Ser.}, pages
  319--343. Cambridge Univ. Press, Cambridge, 1996.

\bibitem{GLW}
{\'E}.~Ghys, R.~Langevin, and P.~Walczak.
\newblock Entropie g\'eom\'etrique des feuilletages.
\newblock {\em Acta Math.}, 160(1-2):105--142, 1988.

\bibitem{Ghys}
{\'E}tienne Ghys.
\newblock Groups acting on the circle.
\newblock {\em Enseign. Math. (2)}, 47(3-4):329--407, 2001.

\bibitem{R.I.Grigorchuk}
R.~I. Grigorchuk.
\newblock Degrees of growth of finitely generated groups and the theory of
  invariant means.
\newblock {\em Izv. Akad. Nauk SSSR Ser. Mat.}, 48(5):939--985, 1984.

\bibitem{Sad}
P.~R. Grossi~Sad.
\newblock A {${\bf Z}\times {\bf Z}$} structurally stable action.
\newblock {\em Trans. Amer. Math. Soc.}, 260(2):515--525, 1980.

\bibitem{Katok}
Anatole Katok.
\newblock Fifty years of entropy in dynamics: 1958--2007.
\newblock {\em J. Mod. Dyn.}, 1(4):545--596, 2007.

\bibitem{KMa}
Piotr Ko{\'s}cielniak and Marcin Mazur.
\newblock Chaos and the shadowing property.
\newblock {\em Topology Appl.}, 154(13):2553--2557, 2007.

\bibitem{LiSc}
Douglas Lind and Klaus Schmidt.
\newblock Symbolic and algebraic dynamical systems.
\newblock In {\em Handbook of dynamical systems, {V}ol.\ 1{A}}, pages 765--812.
  North-Holland, Amsterdam, 2002.

\bibitem{MW}
Dongkui Ma and Min Wu.
\newblock Topological pressure and topological entropy of a semigroup of maps.
\newblock {\em Discrete Contin. Dyn. Syst.}, 31(2):545--557, 2011.

\bibitem{MW14}
R.~Miles and T~Ward.
\newblock Directional uniformities, periodic points and entropy.
\newblock {\em preprint, http://arxiv.org/abs/1411.5295}.

\bibitem{OT11}
Krerley Oliveira and Xueting Tian.
\newblock Non-uniform hyperbolicity and non-uniform specification.
\newblock {\em Trans. Amer. Math. Soc.}, 365(8):4371--4392, 2013.

\bibitem{PS}
C.-E. Pfister and W.~G. Sullivan.
\newblock Large deviations estimates for dynamical systems without the
  specification property. {A}pplications to the {$\beta$}-shifts.
\newblock {\em Nonlinearity}, 18(1):237--261, 2005.

\bibitem{RuelleS}
D.~Ruelle and Ya.~G. Sina{\u\i}.
\newblock From dynamical systems to statistical mechanics and back.
\newblock {\em Phys. A}, 140(1-2):1--8, 1986.
\newblock Statphys 16 (Boston, Mass., 1986).

\bibitem{Ruelle}
David Ruelle.
\newblock Statistical mechanics on a compact set with {$Z^{v}$} action
  satisfying expansiveness and specification.
\newblock {\em Trans. Amer. Math. Soc.}, 187:237--251, 1973.

\bibitem{SSY10}
Kazuhiro Sakai, Naoya Sumi, and Kenichiro Yamamoto.
\newblock Diffeomorphisms satisfying the specification property.
\newblock {\em Proc. Amer. Math. Soc.}, 138(1):315--321, 2010.

\bibitem{Sc15}
F.~M. Schneider.
\newblock Topological entropy of continuous actions of compactly generated
  groups.
\newblock {\em preprint, http://arxiv.org/abs/1502.03980}.

\bibitem{S}
Karl Sigmund.
\newblock On dynamical systems with the specification property.
\newblock {\em Trans. Amer. Math. Soc.}, 190:285--299, 1974.

\bibitem{Si72}
Ja.~G. Sina{\u\i}.
\newblock Gibbs measures in ergodic theory.
\newblock {\em Uspehi Mat. Nauk}, 27(4(166)):21--64, 1972.

\bibitem{Thompson}
Daniel~J. Thompson.
\newblock Irregular sets, the {$\beta$}-transformation and the almost
  specification property.
\newblock {\em Trans. Amer. Math. Soc.}, 364(10):5395--5414, 2012.

\bibitem{Var10}
Paulo Varandas.
\newblock Non-uniform specification and large deviations for weak {G}ibbs
  measures.
\newblock {\em J. Stat. Phys.}, 146(2):330--358, 2012.

\bibitem{SVY13}
Paulo Varandas, N.~Sumi, and Kenichiro Yamamoto.
\newblock Partial hyperbolicity and specification.
\newblock {\em Proc. Amer. Math. Soc.}, 2015,to appear.

\bibitem{SVY15}
Paulo Varandas, N.~Sumi, and Kenichiro Yamamoto.
\newblock Specification and partial hyperbolicity for flows.
\newblock {\em Dyn. Syst.}, 2015,to appear.

\bibitem{Ya09}
Kenichiro Yamamoto.
\newblock On the weaker forms of the specification property and their
  applications.
\newblock {\em Proc. Amer. Math. Soc.}, 137(11):3807--3814, 2009.

\bibitem{CZ14}
Dongmei Zheng and Ercai Chen.
\newblock Bowen entropy for action of amenable groups.
\newblock {\em preprint, http://arxiv.org/abs/1410.4645}.

\end{thebibliography}
\end{document}